\documentclass[12pt]{amsart}
\usepackage{amssymb, amsmath, amsthm}
\usepackage[utf8]{inputenc}
\usepackage{graphicx}
\usepackage{euscript}
\usepackage{dsfont}
\usepackage{color}
\usepackage{xcolor}
\definecolor{myred}{rgb}{0.2,0,0}
\definecolor{myblue}{rgb}{0,0,0.6}
\definecolor{mygreen}{rgb}{0,0.2,0}
\usepackage[colorlinks=true,linkcolor=myred,citecolor=myblue,urlcolor=mygreen]{hyperref}
\usepackage[centering]{geometry}
\usepackage[english]{babel}
\usepackage{bbm}
\usepackage{textcomp}

\usepackage[numbers,sort&compress]{natbib}
\usepackage[sectionbib]{bibunits}
\defaultbibliographystyle{plainnat}

\newcommand{\Q}{\mathbb{Q}}

\newcommand{\les}{\leqslant}
\newcommand{\ges}{\geqslant}

\newtheorem*{conjecture*}{Conjecture}
\newtheorem{theorem}{Theorem}
\newtheorem*{theorem*}{Theorem}
\newtheorem{lemma}{Lemma}

\newtheorem{remark}{Remark}

\DeclareMathOperator{\Sym}{Sym}
\DeclareMathOperator{\meas}{meas}

\begin{document}
	
\title{Variance estimates in Linnik's problem} 
\author{Andrei Shubin}
\address{Department of Mathematics, Caltech, 1200 E. California Blvd., Pasadena, CA, 91125, USA}
\email{\href{mailto:ashubin@caltech.edu}{ashubin@caltech.edu}}

\maketitle

\begin{abstract}
	We evaluate the variance of the number of lattice points in a small randomly rotated spherical ball on a surface of 3-dimensional sphere centered at the origin. Previously, Bourgain, Rudnick, and Sarnak showed conditionally on the Generalized Lindel{\"o}f Hypothesis that the variance is bounded from above by $\sigma(\Omega){N_n}^{1+\varepsilon}$, where $\sigma(\Omega)$ is the area of the ball $\Omega$ on the unit sphere, $N_n$ is the total number of solutions of Diophantine equation $x^2 + y^2 + z^2 = n$. Assuming the Grand Riemann Hypothesis and using the moments method of Soundararajan and Harper, we establish the upper bound of the form $c\sigma(\Omega) N_n$, where $c$ is an absolute constant. This bound is of the conjectured order of magnitude.
\end{abstract}

\section{Introduction}
\label{sec1}

The problem of placing points on the surface of the sphere in a uniform and well-separated way is of great interest in coding theory, tomography, and other areas (see~\cite{Sloane}). We consider the classical example of such a set, namely the integer solutions of the equation $x^2 + y^2 + z^2 = n$. Each solution can be represented by the point $(x, y, z)$ on the sphere with the center at the origin and radius $\sqrt{n}$. We know the upper bound for the number of solutions $N_n \ll n^{1/2 + o(1)}$ and, if $n \ne 0,4,7 \ \text{mod} \ 8$ (which is equivalent to the existence of \textit{primitive} lattice point: $\gcd(x_1, x_2, x_3) = 1$), the lower bound $N_n \gg n^{1/2 - o(1)}$. The conjecture of Linnik states that these points are equidistributed as $n \to +\infty$. Linnik proved it using his ergodic method assuming the Grand Riemann Hypothesis~\cite{Linnik_1968}. Unconditional proof was given by Duke \cite{Duke_1988, Duke_1990}, and Golubeva and Fomenko \cite{Golubeva_Fomenko} after a breakthrough of Iwaniec \cite{Iwaniec_1987}.

We denote by $\Omega_R (\mathbf{x})$ the spherical ball on the surface of the sphere with a center at $\mathbf{x}$. That is the set $\{ \mathbf{y} \in S^2 : \text{dist}(\mathbf{x}, \mathbf{y}) < R  \}$, where $\text{dist}(\mathbf{x}, \mathbf{y})$ is a spherical distance between the points $\mathbf{x}$ and $\mathbf{y}$, $R = R(n)$ is a spherical radius of the ball. We denote $\sigma(\Omega_R)$ the area of the ball normalized so that $\sigma(S^2) = 4\pi$, and by $Z = Z(n; \Omega_R (\mathbf{x}))$ the number of lattice points inside $\Omega_R (\mathbf{x})$. The variance over all the balls is
$$
V \bigl(n; \Omega_R (\mathbf{x})\bigr) := \int_{SO(3)} \biggl| Z(n, g\Omega_R) - \frac{\sigma(\Omega_R)}{\sigma(S^2)} N_n \biggr|^2  d\mu,
$$ where $ d\mu$ is the Haar probability measure on the rotations of the sphere $SO(3)$. The variance does not depend on $\mathbf{x}$.

This object was considered by Bourgain, Rudnick, and Sarnak in \cite{BRS, BRS_first} along with a few other spatial statistics such as electrostatic energy, Ripley's function, nearest neighbour statisitcs, and covering radius. They made the following
\begin{conjecture*}[Bourgain, Rudnick, Sarnak, \cite{BRS}] Let $\Omega_R$ be a sequence of spherical balls, or annuli, and $\delta > 0$ be a given number. If $N_n^{-1+\delta} \ll \sigma(\Omega_R) \ll N_n^{-\delta}$ as $n \to +\infty$, $n \ne 0,4,7 \pmod{8}$, then
	$$
	\int_{SO(3)} \biggl| Z(n; g\Omega_R) - \frac{\sigma(\Omega_R)}{\sigma(S^2)} N_n \biggr|^2  d\mu \sim \frac{\sigma (\Omega_R)}{\sigma(S^2)} N_n.
	$$
\end{conjecture*} In their paper Bourgain, Rudnick, and Sarnak obtained the upper bound assuming the Generalized Lindel{\"o}f Hypothesis:   

\begin{theorem*}[Bourgain, Rudnick, Sarnak, \cite{BRS}]
	Let $\Omega_R$ be a sequence of spherical balls, or annuli, and $N_n^{-1+\delta} \ll \sigma(\Omega_R) \ll N_n^{-\delta}$ for given $\delta \in (0, 1)$. Assume the Lindel{\"o}f Hypothesis for standard $GL(2)/\Q$ $L$-functions. Then for squarefree $n \ne 7 \pmod{8}$, we have
	$$
	\int_{SO(3)} \biggl| Z(n; g\Omega_R) - \frac{\sigma(\Omega_R)}{\sigma(S^2)} N_n \biggr|^2  d\mu \ll_{\varepsilon} n^{\varepsilon} \sigma(\Omega_R) N_n
	$$ for any $\varepsilon > 0$.
\end{theorem*} Recently Humphries and Radziwi{\l\l}~\cite{Humphries} have obtained unconditional results for asymptotics in the case of a thin annulus. In the present paper, we get an upper bound of the conjectured order assuming the Grand Riemann Hypothesis:
\begin{theorem} \label{thm1}
	Assume the Grand Riemann Hypothesis for $GL(2)/\Q$ and Dirichlet $L$-functions. Then for squarefree $n \ne 7 \pmod{8}$, we have
	$$
	\int_{SO(3)} \biggl| Z(n; g\Omega_R) - \frac{\sigma(\Omega_R)}{\sigma(S^2)} N_n \biggr|^2  d\mu \les c \sigma(\Omega_R) N_n, 
	$$ where $\sigma(\Omega_R) = N_n^{-\delta}$ for a fixed $\delta \in (0, 1)$, and $c$ is an absolute constant. 
\end{theorem} 

We express the variance as a sum of Weyl sums of spherical harmonics:
\begin{equation} \label{variance_first}
V(n; \Omega_R) = \sum_{m=1}^{+\infty} h^2 (m) \sum_{j=1}^{2m+1} \bigl| W_{\phi_{j,m}} (n) \bigr|^2, 
\end{equation} where 
$$
\widehat{\mathcal{E}}(n) := \biggl\{ \biggl( \frac{x}{\sqrt{n}}, \frac{y}{\sqrt{n}}, \frac{z}{\sqrt{n}} \biggr) \in S^2: x^2 + y^2 + z^2 = n \biggr\}, \qquad W_{\phi_{j,m}} (n) := \sum_{\mathbf{x} \in \widehat{\mathcal{E}}(n)} \phi_{j,m} ( \mathbf{x}),
$$ and $\{\phi_{j,m}\}$, $-m \les j \les m$, $m = 1, 2, 3, \ldots$ form an orthonormal basis of eigenfunctions of the spherical Laplacian (see~\cite{BRS, Lubotzky}).

The expression in the right hand side of~\eqref{variance_first} can be bounded by a sum of of the product of certain $L$-functions. The famous work of Waldspurger~\cite{Wald_old} gives a relation between Weyl sums $W_{\phi_{j,m}}$ and Fourier coefficients of half-integral weight forms. Another work of Waldspurger~\cite{Waldspurger} provides an identity for the squares of the absolute values of the Weyl sums in terms of the products of $L$-functions. More precisely, the eigenfunctions $\phi_{j,m}$ can be viewed as the the functions on the subspace $D^0 (\mathbb{R})$ of the Hamiltonian quaternion algebra $D(\mathbb{R})$ with trace zero elements. The Weyl sum does not vanish only if $m$ is even and $\phi_{j,m}$ satisfies certain invariance conditions. In this case the Jacquet-Langlands lift~\cite{Mar_White} gives a bijective correspondence between $\phi_{j,m}$ and holomorphic newforms of weight $2m + 2$ and level 2. Precisely,
$$
\bigl| W_{\phi_{j,m}} (n) \bigr|^2 = \frac{C \bigl| \widehat{\mathcal{E}}(n) \bigr|^2}{\sqrt{n} L(1, \chi_{-n})^2} \frac{L\bigl( \frac{1}{2}, f_{j,m} \bigr) L\bigl( \frac{1}{2}, f_{j,m} \otimes \chi_{-n} \bigr)}{L\bigl(1, \Sym^2 f_{j,m}\bigr)},
$$ and so
\begin{equation} \label{variance}
V(n, \Omega_R) \les \sum_{m=1}^{+\infty} h^2 (m) \sum_{j=1}^{2m+1} c \sqrt{n} \frac{L\bigl(\frac{1}{2} , f_{j,m}\bigr) L\bigl(\frac{1}{2} , f_{j,m} \otimes \chi_{-n}\bigr)}{L\bigl(1, \Sym^2 f_{j,m}\bigr)},
\end{equation} where $f_{j,m}$ form a Hecke basis of the space of cusp forms for the congruence subgroup $\Gamma_0 (2)$ of weight $2m + 2$ and trivial nebentypus, $\chi_{-n}$ is the quadratic Dirichlet character corresponding to the extension $\Q(\sqrt{-n})$, $C, c > 0$ are absolute constants. For more details on~\eqref{variance} see~\cite[Lemma~2.16]{Humphries} and~\cite[Section~2]{Boch_Schu_Sarnak}. 

In this way, the problem reduces to estimating the first moment of the product of $L$-functions. We apply the technique developed by Soundararajan~\cite{Sound} and Harper~\cite{Harper} for the moments of zeta function assuming the Riemann Hypothesis. The goal is to get an upper bound for the sum on the right hand side of~\eqref{variance} when $m \gg n^{\delta / 2}$. The desired upper bound is given by
\begin{theorem} \label{thm2}
	Let $m > n^{\varepsilon}$ for fixed $\varepsilon > 0$. Assuming the Grand Riemann Hypothesis for $GL(2)/\Q$ and Dirichlet $L$-functions, we have
	$$
	\sum_{f \in S_{2m+2} (\Gamma_0 (2))} \frac{L\bigl( \frac{1}{2}, f \bigr) L \bigl( \frac{1}{2}, f \otimes \chi_{-n} \bigr)}{L \bigl( 1, \Sym^2 f \bigr)} \les m L(1, \chi_{-n}) \exp \bigl( U(n,m) \bigr),
	$$ where 
	$$
	U(n,m) = C_1 \max \biggl( 1, \frac{1}{2} \frac{\log n}{\log m} \biggr) \exp \biggl(C_2 \max \biggl( 1, \frac{1}{2} \frac{\log n}{\log m} \biggr) \biggr),
	$$ and $C_1$ and $C_2$ are the absolute constants. 
\end{theorem} 

From the definition of $U(m,n)$ it follows that the upper bound is of size $m L(1, \chi_{-n})$ for $m \gg n^{\delta/2} = (\sigma(\Omega_R))^{-1/2}$. When $m$ is small we bound each term in the sum over $f \in S_{2m+2} (\Gamma_0 (2))$ individually instead of applying Theorem~\ref{thm2}.

The paper is organized as follows. In Section~\ref{sec2} we motivate the conjecture of Bourgain, Rudnick, and Sarnak by considering the random uniformly distributed points on the sphere. In Section~\ref{sec3} we deduce Theorem~\ref{thm1} from Theorem~\ref{thm2}. The rest of the paper is devoted to the proof of Theorem~\ref{thm2}. In Section~\ref{sec4} we outline the main ideas of Soundararajan's and Harper's approaches. In Section~\ref{sec5} we prove a few necessary combinatorial lemmas based on Petersson's trace formula. Using these lemmas and Harper's argument, we give a detailed proof of Theorem~\ref{thm2} in Section~\ref{sec6}.

\section{Random Model}
\label{sec2}

In this section we compute the expected value and the variance of random uniformly distributed points $\mathbf{x}_1$, $\ldots$, $\mathbf{x}_{N_n}$ from $\widehat{\mathcal{E}}(n)$ on the sphere inside a ball $\Omega_R (\mathbf{x})$ centered at some fixed point $\mathbf{x}$. Similarly to the previous section, denote by $Z(n; \Omega_R (\mathbf{x}))$ the number of points in the ball. Then $Z$ can be written as $\sum_{i=1}^N \xi_i (\mathbf{x})$, where
$$
\xi_i (\mathbf{x}) = 
\begin{cases}
\displaystyle 1,
\text{if } 
\displaystyle \mathbf{x} \in \Omega_R (\mathbf{x_i});
\\[2mm]
\displaystyle 0,
\text{otherwise}.
\end{cases}
$$ Note that one can think of random point $\mathbf{x}_i$ as of random rotation of the sphere moving $\mathbf{x}$ to $\mathbf{x}_i$. In this section we use renormalized Haar measure
$$
d\tilde\mu = \frac{d\mu}{\sigma(S^2)}, \quad \text{so that} \quad \tilde \sigma(\Omega_R) = \frac{\sigma(\Omega_R)}{\sigma(S^2)}, \quad \tilde \sigma(S^2) = 1.
$$ Then the expected value of the number of points inside the ball can be computed as follows:
$$
E[Z] = \int_{SO(3)} Z(n; g\Omega_R (\mathbf{x}))  d\tilde\mu = \sum_{i=1}^{N_n} E[\xi_i] = \sum_{i=1}^{N_n} \int_{SO(3)} {\mathbf{1}}_{\mathbf{x} \in g\Omega(\mathbf{x}_i)}  d\tilde\mu = \tilde\sigma(\Omega_R) N_n.
$$ The variance is
\begin{multline*}
V(n, \Omega_R (\mathbf{x})) = \int_{SO(3)} \bigl( Z (n, g\Omega_R) - \tilde\sigma (g\Omega_R)  N_n \bigr)^2  d\tilde\mu = \\ 
E[Z^2] - 2 \tilde\sigma(\Omega_R) N_n E[Z] + \bigl(\tilde\sigma(\Omega_R)  N_n \bigr)^2 = E[Z^2] - (\tilde\sigma(\Omega_R)  N_n)^2. 
\end{multline*} Next,
$$
E[Z^2] = \sum_{i=1}^{N_n} E[\xi_i^2] + \sum_{\substack{1 \les i, j \les N_n \\ i \ne j}} E[\xi_i \xi_j] = \tilde\sigma(\Omega_R) N_n + \tilde\sigma (\Omega_R)^2 N_n(N_n-1)
$$ since $\xi_i$ and $\xi_j$ are independent. Thus,
$$
V(n; \Omega_R) = \tilde\sigma(\Omega_R) N_n - \tilde\sigma (\Omega_R)^2 N_n \sim \frac{\sigma(\Omega_R)}{\sigma(S^2)} N_n 
$$ as soon as $\tilde\sigma(\Omega_R) = o(1)$.

\section{Variance Estimate. Proof of Theorem 1}
\label{sec3}

For completeness we first prove formula~\eqref{variance_first}. For fixed $n$ we choose a spherical ball $\Omega_R$ of spherical radius $R$ on the sphere of radius $\sqrt{n}$ and denote the area of the ball by $\sigma (\Omega_R)$. We will further need the notion of the point-pair invariant $K_R (\mathbf{x}, \mathbf{y})$. It can be defined as follows:
$$
K_R (\mathbf{x}, \mathbf{y}) =
\begin{cases}
\displaystyle 1,
\text{ if }
\displaystyle \mathbf{y} \in \Omega_R (\mathbf{x});
\\[2mm]
\displaystyle 0,
\text{ otherwise}.
\end{cases}
$$ Clearly it is a function of spherical distance between $\mathbf{x}$ and $\mathbf{y}$. The properties of this function and its Fourier expansion were described in details in~\cite{Lubotzky}. Here we follow their approach. 

The Fourier expansion has the form 
$$
K_R (\mathbf{x}, \mathbf{y}) = \sum_{m=0}^{+\infty} h_R (m) \sum_{j=1}^{2m+1} \phi_{j,m} (\mathbf{x}) \overline{\phi_{j,m} (\mathbf{y})},
$$ where $\phi_{j,m}$ form an orthonormal basis with the inner product
$$
\langle f, g \rangle = \int_{SO(3)} f(\mathbf{x}) \overline{g(\mathbf{x})}  d\mu(\mathbf{x}),
$$ and the coefficients $h_R (m)$ are given by Selberg-Harish-Chandra transform (see~\cite{Lubotzky} or~\cite{Humphries}):
$$
h_R (m) = 2 \pi \int_0^{\pi} P_m (\cos{\theta}) k_R (\cos{\theta}) \sin{\theta} d\theta.
$$ Here
$$
P_m (x) = \frac{1}{2^m m!} \frac{d^m}{dx^m} (x^2 - 1)^m
$$ are the Legendre polynomials, and
$$
k_R (\cos{\theta}) =
\begin{cases}
\displaystyle 1,
\text{if }
\displaystyle |\theta| \les R;
\\[2mm]
\displaystyle 0,
\text{otherwise}
\end{cases}
$$ is the point-pair invariant written in the spherical coordinates.

As in the previous section we denote the lattice points as $\mathbf{x}_1, \ldots, \mathbf{x}_{N_n}$. We next compute the variance as follows:
\begin{multline} \label{variance_expansion}
V\bigl( n, \Omega_R (\mathbf{x}) \bigr) = \int_{SO(3)} \biggl( \sum_{i=1}^{N_n} K_R (\mathbf{z}, \mathbf{x}_i) - \frac{\sigma(\Omega_R)}{\sigma(S^2)} N_n \biggr)^2  d\mu(\mathbf{z}) = \\ 
\int_{SO(3)} \biggl( \sum_{i=1}^{N_n} \sum_{m=0}^{+\infty} h_R (m) \sum_{j=1}^{2m+1} \phi_{j,m} (\mathbf{z}) \overline{\phi_{j,m} (\mathbf{x}_i)} - \frac{\sigma(\Omega_R)}{\sigma(S^2)} N_n \biggr)^2 d\mu(\mathbf{z}).
\end{multline} Note that
$$
\sum_{i=1}^{N_n} h_R (0) \phi_{0,0} (\mathbf{z}) \overline{\phi_{0,0}(\mathbf{x}_i)} = \frac{\sigma(\Omega_R)}{\sigma(S^2)} N_n.
$$ Indeed,
\begin{gather*}
h_R (0) = 2\pi \int_0^{\pi} k_R (\cos \theta) \sin \theta d\theta = 2\pi \int_0^{R} \sin \theta d\theta = 2\pi(1 - \cos R) = \sigma(\Omega_R), \\
\phi_{0,0} (\mathbf{z}) = \frac{1}{2\sqrt{\pi}}.
\end{gather*} Next, in order to be able to change the orders of summation and integration for the Fourier series we introduce the smooth version of point-pair invariant. It is given by a convolution $K_R \ast K_{\rho} (\mathbf{x}, \mathbf{y})$ for some small spherical radius $\rho$ such that $0 < \rho \ll R$. We denote the corresponding spherical ball by $\Omega_{\rho}$. We clearly have
$$
\frac{K_{R-\rho} \ast K_{\rho} (\mathbf{x}, \mathbf{y})}{\sigma(\Omega_{\rho})} \les K_R (\mathbf{x}, \mathbf{y}) \les \frac{K_{R+\rho} \ast K_{\rho} (\mathbf{x}, \mathbf{y})}{\sigma(\Omega_{\rho})}.
$$ Denote the right hand side and the left hand side of these inequalities as $\tilde K_{R+\rho}$ and $\tilde K_{R-\rho}$, and their Selberg-Harish-Chandra transforms by $\tilde h_{R+\rho}$ and $\tilde h_{R-\rho}$ correspondingly. Then
\begin{equation} \label{bounds_conv}
\tilde h_{R+\rho} (m) = \frac{h_{R+\rho} (m) h_{\rho}(m)}{\sigma(\Omega_{\rho})}, \qquad  \tilde h_{R-\rho} (m) = \frac{h_{R-\rho} (m) h_{\rho}(m)}{\sigma(\Omega_{\rho})}.
\end{equation}

By~\cite[Lemma~2.12]{Humphries} the Fourier expansion of the convolution for any $\rho_1$, $\rho_2 > 0$
$$
K_{\rho_1} \ast K_{\rho_2} (\mathbf{x}, \mathbf{y}) = \sum_{m=0}^{+\infty} h_{\rho_1}(m) h_{\rho_2} (m) \sum_{j=1}^{2m+1} \phi_{j,m} (\mathbf{x}) \overline{ \phi_{j,m} (\mathbf{y})}
$$ converges absolutely and uniformly. By~\eqref{variance_expansion} and the triangle inequality,
\begin{multline} \label{triangle_ineq}
V(n, \Omega_R (\mathbf{x})) \les \int_{SO(3)} \biggl( \sum_{i=1}^{N_n} \bigl( K_R (\mathbf{z}, \mathbf{x}_i) - \tilde K_{R+\rho} (\mathbf{z}, \mathbf{x}_i) \bigr) \biggr)^2  d\mu(\mathbf{z}) + \\
\int_{SO(3)} \biggl( \sum_{i=1}^{N_n} \tilde K_{R+\rho} (\mathbf{z}, \mathbf{x}_i) - \sigma(\Omega_{R+\rho}) N_n \biggr)^2  d\mu(\mathbf{z}) + \\
\int_{SO(3)} \biggl( \sigma(\Omega_{R+\rho}) N_n - \sigma(\Omega_R) N_n \biggr)^2 d\mu =:
V_1 + V_2 + V_3.
\end{multline}

The easiest term to estimate is $V_3$. We get:
\begin{equation} \label{V_3}
	V_3 \ll N_n^2 (\cos(R) - \cos(R+\rho)) \ll N_n^2 \rho R.
\end{equation}

Next, we estimate the contribution from $V_2$. Now since the Fourier series for $\tilde K_{R+\rho}$ converges absolutely and uniformly we can switch the order of summation and integration. Thus, we get
\begin{multline*}
	V_2 \les \sum_{m_1 = 1}^{+\infty} \sum_{m_2 = 1}^{+\infty} \tilde h_{R+\rho} (m_1) \tilde h_{R+\rho}(m_2) \sum_{j_1 = 1}^{2m_1 + 1} \sum_{j_2 = 1}^{2m_2 + 1}  W_{\phi_{j_1, m_1}} (n) \overline{ W_{\phi_{j_2, m_2}} (n)} \\
	\int_{SO(3)} \phi_{j_1, m_1} (\mathbf{z}) \overline{ \phi_{j_2, m_2} (\mathbf{z})} d\mu(\mathbf{z}),
\end{multline*} where
$$
W_{\phi_{j, m}} (n) = \sum_{i=1}^{N_n} \phi_{j, m} (\mathbf{x}_i).
$$ By the orthogonality of spherical harmonics we deduce
\begin{equation} \label{V_2}
V_2 \les \sum_{m=1}^{+\infty} \tilde h_{R+\rho}^2 (m) \sum_{j=1}^{2m+1} \bigl|W_{\phi_{j,m}} (n)\bigr|^2. 
\end{equation}

To estimate the contribution from $V_1$, we apply trivial inequality
$$
V_1 \les \int_{SO(3)} \biggl( \sum_{i=1}^{N_n} \bigl( \tilde K_{R-\rho} (\mathbf{z}, \mathbf{x}_i) - \tilde K_{R+\rho} (\mathbf{z}, \mathbf{x}_i) \bigr) \biggr)^2  d\mu(\mathbf{z}).
$$ Now using the Fourier expansions of $\tilde K_{R+\rho}$ and $K_{R-\rho}$ by the same argument we obtain
\begin{multline} \label{V_1}
	V_1 \les \int_{SO(3)} \biggl( \sigma(\Omega_{R+\rho}) N_n - \sigma(\Omega_{R-\rho}) N_n \biggr)^2 d\mu + \\
	\sum_{m=1}^{\infty} \bigl| \tilde h_{R+\rho} (m) - \tilde h_{R-\rho} (m) \bigr|^2 \sum_{j=1}^{2m+1} |W_{\phi_{j,m}} (n)|^2 \les
	N_n^2 (\rho R)^2 + \sum_{m=1}^{\infty} \tilde h_{R+\rho}^2 (m) \sum_{j=1}^{2m+1} |W_{\phi_{j,m}} (n)|^2
\end{multline} since clearly $|\tilde h_{R+\rho}(m) - \tilde h_{R-\rho} (m)|^2 \ll \tilde h_{R+\rho}^2 (m)$.

Thus, combining together~\eqref{triangle_ineq}, \eqref{V_3}, \eqref{V_2}, and \eqref{V_1}, we get
\begin{equation} \label{variance_ineq_2}
V(n, \Omega_R (\mathbf{x})) \les \sum_{m=1}^{+\infty} \tilde h_{R+\rho}^2 (m) \sum_{j=1}^{2m+1} \bigl|W_{\phi_{j,m}} (n)\bigr|^2 + O(N_n^2 \rho R).
\end{equation}

Now we move to the first moment of the product of $L$-functions. Recall that if $m$ is even and $\phi_{j,m}$ is invariant under the action of a certain subgroup, the Jacquet-Langlands lift gives a bijective correspondence between the basis of spherical harmonics $\phi_{j,m}$ of order $m$ and Hecke basis of holomorphic cusp forms of $\Gamma_0 (2)$ of weight $2m+2$, level 2, and trivial nebentypus (see~\cite[Lemma~2.16]{Humphries} and~\cite[Section~2]{Boch_Schu_Sarnak}). The remaining Weyl sums vanish. So we have
$$
\bigl|W_{\phi_{j,m}} (n)\bigr|^2 \les \frac{c\sqrt{n} L\bigl( \frac{1}{2}, f_{j,m} \bigr) L\bigl( \frac{1}{2}, f_{j,m} \otimes \chi_{-n} \bigr)}{L\bigl( 1, \Sym^2 f_{j,m} \bigr)}
$$ with an absolute constant $c > 0$ independent of $m$, $n$ and $\phi_{j,m}$. So we need to evaluate the expression
$$
\sum_{m=1}^{+\infty} \tilde h_{R+\rho}^2 (m) \sum_{j=1}^{2m+1} \frac{c\sqrt{n} L\bigl( \frac{1}{2}, f_{j,m} \bigr) L\bigl( \frac{1}{2}, f_{j,m} \otimes \chi_{-n} \bigr)}{L\bigl( 1, \Sym^2 f_{j,m} \bigr)}.
$$ We do it separately for small and large values of $m$. Let
\begin{equation} \label{def_M}
M := \frac{1}{\sqrt{\sigma(\Omega_R)}} \exp \biggl( -\frac{2 \log n}{\log \log n} \biggr).
\end{equation} 

First, consider the case $m \les M$. This case is easier since we only use pointwise bounds for the product of $L$-functions and Selberg-Harish-Chandra transform $\tilde h_{R+\rho}^2 (m)$. Applying the conditional upper bounds for $L(1/2, f)$, $L(1/2, f \otimes \chi_{-n})$~\cite[Corollary~1.1]{Chandee}, and conditional lower bounds for $L(1, \Sym^2 f)$~\cite{Hoffstein_Lockhart}, we get 
\begin{equation} \label{pointwise_product}
\frac{L\bigl( \frac{1}{2}, f_{j,m} \bigr) L\bigl( \frac{1}{2}, f_{j,m} \otimes \chi_{-n} \bigr)}{L\bigl( 1, \Sym^2 f_{j,m} \bigr)} \ll \exp\biggl(  \frac{\log (m^4 n^2)}{\log \log n} \biggr).
\end{equation} Next, we show that
$$
\tilde h_{R+\rho} (m) \ll \sigma(\Omega_R).
$$ By Hilb's formula (see \cite{Humphries} or \cite[Theorem~8.21.6]{Szego}) we have the following connection between polynomials $P_m (\cos \theta)$ and Bessel function $J_0 (x)$:
\begin{equation} \label{Hilb}
P_m (\cos \theta) = \sqrt{\frac{\theta}{\sin \theta}} J_0 \biggl( \biggl( m + \frac{1}{2} \biggr) \theta \biggr) + 
\begin{cases}
\displaystyle O(\theta^2) 
& \text{for } \displaystyle m \les \frac{1}{\theta},
\\[2mm]
\displaystyle O_{\varepsilon} \bigl( \frac{\sqrt{\theta}}{m^{3/2}} \bigr)
& \text{for } \displaystyle m \ges \frac{1}{\theta} \ges \frac{1}{\pi - \varepsilon}
\end{cases}
\end{equation} for fixed $\varepsilon > 0$, $0 < \theta < \pi - \varepsilon$. The function $J_0$ satisfies the asymptotic formula
\begin{equation} \label{Bessel}
J_0 (x) = 
\begin{cases}
\displaystyle 1 + O(x^2)
& \text{for } \displaystyle |x| \les 1,
\\[2mm]
\displaystyle \sqrt{\frac{2}{\pi |x|}} \cos \bigl( |x| - \frac{\pi}{4} \bigr) + O\bigl( \frac{1}{|x|^{3/2}} \bigr) 
& \text{for } \displaystyle |x| \ges 1.
\end{cases}
\end{equation} Since $m \ll (\sigma(\Omega_R))^{-1/2} \sim R^{-1}$, by~\eqref{Hilb} we have
\begin{multline*}
	h_R (m) = 2\pi \int_0^{R} P_m (\cos \theta) \sin \theta d\theta =
	\int_0^{R} \sqrt{\frac{\theta}{\sin \theta}} J_0 \biggl( \biggl( m + \frac{1}{2} \biggr) \theta \biggr) \sin \theta d\theta + \\
	O \biggl( \int_0^{R} \theta^2 \sin \theta d\theta \biggr),
\end{multline*} and since $(m + 1/2)\theta \les (M + 1/2) R \ll 1$, by~\eqref{Bessel}, we get
\begin{multline*}
h_R (m) \ll \int_0^{R} \sqrt{\theta \sin \theta} d\theta + O\biggl( \int_0^{R} m^2 \theta^2 \sqrt{\theta \sin \theta} d\theta \biggr) + O \biggl( \int_0^{R} \theta^2 \sin \theta d\theta \biggr) \ll \\ 
R^2 + O(m^2 R^4) + O(R^4) \ll \sigma(\Omega_R).
\end{multline*} Then by~\eqref{bounds_conv},
\begin{equation} \label{Chandra_bound}
\tilde h_{R+\rho} (m) = \frac{h_{R+\rho} (m) h_{\rho} (m)}{\sigma(\Omega_{\rho})} \ll |h_R (m)| \ll \sigma(\Omega_R).
\end{equation} Finally,~\eqref{pointwise_product} and~\eqref{Chandra_bound} give
\begin{multline} \label{first_part_bound}
\sum_{m \les M} \tilde h_{R+\rho}^2 (m) \sum_{j=1}^{2m+1} \frac{c\sqrt{n} L\bigl( \frac{1}{2}, f_{j,m} \bigr) L\bigl( \frac{1}{2}, f_{j,m} \otimes \chi_{-n} \bigr)}{L\bigl( 1, \Sym^2 f_{j,m} \bigr)} \ll \\ 
\sqrt{n} \exp \biggl( \frac{\log (M^4 n^2)}{\log \log n} \biggr) \sum_{m \les M} m\tilde h_{R+\rho}^2 (m) \ll  \sqrt{n} \exp \biggl( \frac{\log (M^4 n^2)}{\log \log n} \biggr) \sigma(\Omega_R)^2 M^2.
\end{multline}

Next, consider the case $m > M$. Now we estimate the product of $L$-functions on average over $j$, using the first moment bound given by Theorem~\ref{thm2}:
\begin{equation} \label{product_average}
\sum_{j=1}^{2m+1} \frac{ L\bigl( \frac{1}{2}, f_{j,m} \bigr) L\bigl( \frac{1}{2}, f_{j,m} \otimes \chi_{-n} \bigr)}{L\bigl( 1, \Sym^2 f_{j,m} \bigr)} \ll m L(1, \chi_{-n}) \exp \bigl( U(n,M)  \bigr).
\end{equation} We also need an average estimate for $\tilde h_{R+\rho}^2 (m)$. We show the following bound
\begin{equation} \label{Chandra_average}
\sum_{m=1}^{+\infty} m \tilde h_{R+\rho}^2 (m) \ll \sigma (\Omega_R).
\end{equation} Indeed,
\begin{multline*}
	\sigma(\Omega_{R+\rho} (\mathbf{x})) \gg \int_{SO(3)} \biggl( \tilde K_{R+\rho} (\mathbf{x}, \mathbf{z}) \biggr)^2 d\mu(\mathbf{z}) = \\
	\int_{SO(3)} \biggl( \sum_{m=0}^{+\infty} \tilde h_{R+\rho} (m) \sum_{j=1}^{2m+1} \phi_{j,m} (\mathbf{z}) \overline{\phi_{j,m} (\mathbf{x})} \biggr)^2 d\mu(\mathbf{z}) = \\ 
	\sum_{\substack{m_1, m_2 \\ j_1, j_2}} \tilde h_{R+\rho}(m_1) \tilde h_{R+\rho} (m_2) \phi_{j_1, m_1} (\mathbf{x}) \overline{ \phi_{j_2, m_2} (\mathbf{x})} \int_{SO(3)} \phi_{j_1, m_1} (\mathbf{z}) \overline{ \phi_{j_2, m_2} (\mathbf{z})} d\mu(\mathbf{z}) = \\ \sum_{m=0}^{+\infty} \tilde h_{R+\rho}^2 (m) \sum_{j=1}^{2m+1} |\phi_{j,m} (\mathbf{x})|^2 = \sum_{m=0}^{+\infty} \tilde h_{R+\rho}^2 (m) \frac{2m+1}{2\pi}.
\end{multline*} For the last identity see~\cite[formula~(2.11)]{Lubotzky}. Now~\eqref{Chandra_average} trivially follows from
$$
\sigma(\Omega_R (\mathbf{x})) \sim \sigma(\Omega_{R+\rho} (\mathbf{x})), \qquad \text{and} \qquad \sum_{m=0}^{+\infty} \tilde h_{R+\rho}^2 (m) \frac{2m+1}{2\pi} \gg \sum_{m=0}^{+\infty} \tilde h_{R+\rho}^2 (m) m.
$$

Combining together~\eqref{first_part_bound}. \eqref{product_average}, and~\eqref{Chandra_average}, we get
\begin{multline*}
\sum_{m = 1}^{+\infty} \tilde h_{R+\rho}^2 (m) \sum_{j=1}^{2m+1} \frac{c\sqrt{n} L\bigl( \frac{1}{2}, f_{j,m} \bigr) L\bigl( \frac{1}{2}, f_{j,m} \otimes \chi_{-n} \bigr)}{L\bigl( 1, \Sym^2 f_{j,m} \bigr)} \ll \\ \sqrt{n} \exp \biggl( \frac{\log (M^4 n^2)}{\log \log n} \biggr) \sigma(\Omega_R)^2 M^2 + 2\sqrt{n} L(1, \chi_{-n}) \exp \bigl( U(n,M) \bigr) \sigma(\Omega_R).
\end{multline*} Hence, by~\eqref{variance_ineq_2},
\begin{multline*}
	V(n, \Omega_R) \ll \sqrt{n} \sigma(\Omega_R) \biggl( \sigma(\Omega_R) M^2 \exp \biggl( \frac{\log (M^4 n^2)}{\log \log n}  \biggr) + L(1, \chi_{-n}) \exp\bigl( U(n, M)  ) \biggr) + \\
	O(N_n^2 \rho R).
\end{multline*} With the choice of $M$ given by~\eqref{def_M} and the restriction $(\sigma(\Omega_R))^{-1} \ll \sqrt{n}$ we have
$$
\sigma(\Omega_R) M^2 \exp \biggl( \frac{\log (M^4 n^2)}{\log \log n}  \biggr) \ll \exp\biggl( -\frac{\log n}{\log \log n} \biggr).
$$ Next, note that since $\sigma(\Omega_R) = N_n^{-\delta}$ for some $\delta \in (0, 1)$, the bound of Theorem~\ref{thm2} gives $U(n, M) \asymp 1$. Choosing $\rho$ sufficiently small we finally get
$$
V(n, \Omega_R) \ll \sqrt{n} \sigma(\Omega_R) L(1, \chi_{-n}) \ll \sigma(\Omega_R) N_n
$$ as desired.

\section{Moment Estimate. Main Ideas}
\label{sec4}

In this section, we give a sketch of the proof of Theorem~\ref{thm2} using the ideas of Soundararajan~\cite{Sound} and Harper~\cite{Harper}. The main goal in their works was to obtain a sharp upper bound for the moments of the Riemann zeta function
$$
M_k (T) = \int_0^T \bigl|\zeta(\frac{1}{2} + it)\bigr|^{2k} dt
$$ for all positive real $k$ assuming the Riemann Hypothesis. In~\cite{Sound} it was shown that $M_k (T) \ll_{k, \varepsilon} T(\log T)^{k^2 + \varepsilon}$ for any $\varepsilon > 0$. In~\cite{Harper} this estimate was sharpened to $M_k (T) \ll_k T(\log T)^{k^2}$.  

The idea of Soundararajan's work goes back to Selberg who studied the distribution of $\log \zeta(1/2 + it)$. His method works very well for imaginary part of the logarithm but leads to complications in the case of the real part because of zeros lying very close to the critical line. To get an upper bound for $\log |\zeta(1/2 + it)|$, one does not need to explore the contribution from zeros since it is essentially negative. The following upper bound (see the main Proposition in~\cite{Sound}) holds true on RH:
\begin{equation} \label{log_zeta}
\log \bigl| \zeta(\frac{1}{2} + it) \bigr| \les \text{Re} \sum_{n \les x} \frac{\Lambda(n)}{n^{1/2 + it} \log(n)} + \frac{3}{4} \frac{\log T}{\log x} + R(x)
\end{equation} for any $t \in [T, 2T]$, $2 \les x \les T^2$. Here $R(x)$ corresponds to lower-order terms. The expression for $M_k (T)$ can be rewritten as
$$
\int_T^{2T} \bigl| \zeta(\frac{1}{2} + it) \bigr|^{2k} dt = 2k \int_{-\infty}^{+\infty} e^{2kV} \meas \biggl\{ t \in [T, 2T] : \log \bigl|\zeta(\frac{1}{2} + it)\bigr| \ges V  \biggr\} dV.
$$ In order to get an upper bound  for $M_k (T)$ one needs to obtain an appropriate upper bound for the measure inside the integral. This measure depends on the size of~$V$. This was achieved in Soundararajan's work by exploring the joint behaviour of short and long Dirichlet polynomials
$$
\sum_{p \les z} \frac{1}{p^{1/2 + it}}, \qquad \sum_{z < p \les x} \frac{1}{p^{1/2 + it}}
$$ as $t$ varies between $T$ and $2T$. Rougly speaking it was shown that for most $t$ the absolute values of the real parts of these polynomials cannot be too large. 

A slightly different approach leading to a sharper upper bound was elaborated in Harper's work. The long Dirichlet polynomial splits into the sum of many shorter ones of the form
$$
\sum_{T^{\beta_{i-1}} < p \les T^{\beta_i}} \frac{1}{p^{1/2 + it}}, \qquad 0 = \beta_0 < \beta_1 < \ldots < \beta_{I-1} < \beta_I,
$$ and the better bound is obtained by exploring the joint distribution of the real parts of these polynomials. 

Both methods work well in the case of more general class of $L$-functions. Some variations of these methods were applied to products of automorphic $L$-fucntions~\cite{Milinovich} and averages over fundmental discriminants of central values of quadratic twists~\cite{Sono}. 

We will follow the Harper's approach to prove Theorem~\ref{thm2}. The expression analogous to the right hand side of~\eqref{log_zeta} can be also obtained for a general class of $L$-functions (see~\cite{Chandee} and~Lemma~\ref{lem1} below). This expression can also be written as a sum of a number of Dirichlet polynomials of the form
$$
\sum_{x^{1/3} < p \les x} \frac{\lambda_f(p)}{\sqrt{p}} + \sum_{x^{1/9} < p \les x^{1/3}} \frac{\lambda_f (p)}{\sqrt{p}} + \ldots + \sum_{x^{3^{-I}} < p \les x^{3^{1-I}}} \frac{\lambda_f (p)}{\sqrt{p}} + \sum_{1 < p \les x^{3^{-I}}} \frac{\lambda_f (p)}{\sqrt{p}}.
$$ Here $f \in S_{2m+2} (\Gamma_0 (2))$ and its Fourier coefficients normalized so that $\lambda_f (1) = 1$. More precisely, we will estimate the first moment as 
\begin{multline} \label{expression_for_moment}
\mathop{{\sum}^{h}}_{f \in S_{2m+2} ( \Gamma_0 (2) )} \exp \bigl( \log \bigl| L_1 (f) L_2 (f) \bigr| \bigr) = \\ 
\mathop{{\sum}^{h}}_{f \in S_{2m+2} ( \Gamma_0 (2) )} \exp \biggl( \sum_{1 < p \les x^{3^{-I}}} \frac{\lambda_f (p)}{\sqrt{p}} \biggr) \prod_{i=1}^I \exp \biggl( \sum_{x^{3^{-i}} < p \les x^{3^{1-i}}} \frac{\lambda_f (p)}{\sqrt{p}} \biggr) + R_m (X),
\end{multline} where there precise form of the error term $R_m (x)$ is given in Lemma~\ref{lem1}. Here we use the notation 
\begin{gather*}
L_1 (f) := L \bigl(\frac{1}{2}, f\bigr), \qquad L_2 (f) := L\bigl(\frac{1}{2}, f \otimes \chi_{-d}\bigr), \\
\mathop{{\sum}^{h}}_{f \in S_{2m+2} ( \Gamma_0 (2) )} (.) := \sum_{f \in S_{2m+2}(\Gamma_0 (2))} \frac{(.)}{L(1, \Sym^2 f)}.
\end{gather*}

We will show that for most cusp forms $f \in S_{2m+2} (\Gamma_0 (2))$ one can obtain an appropriate upper bound directly analyzing the truncated Taylor series of each factor $\exp ( \ldots )$. The main tool for getting sharp upper bounds is a multidimensional analogue of Petersson trace formula, which gives the asymptotic formula for the sums of the form
$$
\mathop{{\sum}^{h}}_{f \in S_{2m+2} ( \Gamma_0 (2) )} \lambda_f (p_1) \ldots \lambda_f (p_r).
$$ The error term in the Taylor expansion can only be large for a small amount of cusp forms $f$. We show that the contribution from these ``exceptional'' forms to~\eqref{expression_for_moment} is small using Rankin's trick and Markov's inequality. The details are explicated in Section~\ref{sec6}.

For further convenience we change the notation for the weight of the cusp forms $k := 2m+2$ and the modulus of Dirichlet characters $d : = n$.

\section{Moment Estimate. Auxiliary Lemmas}
\label{sec5}

To get the bound of Theorem~\ref{thm2}, we need two key ingredients, the inequality analogous to~\eqref{log_zeta} for $GL(2)$ $L$-functions and the analogue of mean-value estimate for the Fourier coefficients of the corresponding holomorphic cusp forms. The first ingredient is given by

\begin{lemma} \label{lem1}
	Assume the Grand Riemann Hypothesis and suppose $f \in S_k ( \Gamma_0 (2) )$ is a holomorphic Hecke cusp form. Then for any $x \ges 2$ we have
	\begin{multline*}
	\log \bigl| L_1 (f) L_2 (f) \bigr| \les \sum_{p \les x} \frac{\lambda_f (p) [1 + \chi_{-d} (p)]}{\sqrt{p}} W (p, x) + \\
	\sum_{p \les \sqrt{x}} \frac{\lambda_f (p^2) - 1}{p} W (p^2, x) + c_0 + \frac{3}{4} \frac{\log (d^2 k^4)}{\log x},
	\end{multline*} where $c_0 \les 2$ is an absolute constant,
	\begin{gather*}
	W (n, x) = n^{-\lambda_0 / \log x} \frac{\log (x/n)}{\log x}, \\
	e^{-\lambda_0} = \lambda_0 + \frac{\lambda_0^2}{2} \Longrightarrow \lambda_0 = 0.4912\ldots .
	\end{gather*}
\end{lemma} For the proof see~\cite[Theorem~2.1]{Chandee}. \\

The second key ingredient is the multidimensional analogue of Petersson's trace formula. We first state the classical version:

\begin{lemma}[\textit{Petersson's formula}] \label{lem2} Let $k$ be fixed positive integer, and $m$ and $n$ are natural numbers with $mn \les k^2 / 10^4$. Then
	$$
	\mathop{{\sum}^{h}}_{f \in S_k ( \Gamma_0 (2) )} \lambda_f (n) \lambda_f (m) = \frac{k-1}{2\pi^2} \mathbf{1}_{n=m} + E_k, \qquad E_k \les c_1 k e^{-k},
	$$ where $c_1$ is an absolute constant and
	$$
	\mathop{{\sum}^{h}}_{f \in S_k ( \Gamma_0 (2) )} \lambda_f (n) \lambda_f (m) := \sum_{f \in S_k ( \Gamma_0 (2) )} \frac{1}{L(1, \Sym^2 f)} \lambda_f(n) \lambda_f (m).
	$$
\end{lemma} For the proof see~\cite[Lemma~2.1]{Rudnick_Sound}. \\

\begin{lemma}[\textit{multidimensional Petersson's formula}] \label{lem3} 
	Suppose $N = p_1^{\beta_1} \ldots p_r^{\beta_r} \les \sqrt{k}$, where $p_1$, $\ldots$, $p_r$ are distinct primes. Then
	$$ 
	1) \mathop{{\sum}^{h}}_{f \in S_k ( \Gamma_0 (2) )} \prod_{j=1}^r {\lambda_f} (p_j)^{\beta_j} [1 + \chi_{-d} (p_j)]^{\beta_j} = F(N) + E(N),
	$$ where
	$$
	F(N) =
	\begin{cases} 
	\displaystyle 0, 
	\text{ if }
	\displaystyle \beta_j 
	\text{ is odd for at least one j};
	\\[2mm]
	\displaystyle \frac{k-1}{2\pi^2} \prod_{j=1}^r (1+\chi_{-d} (p_j))^{\beta_j} \frac{2}{\beta_j + 2} \binom{\beta_j}{\beta_j / 2}, 
	\text{ otherwise}, 
	\end{cases}
	$$ and
	\begin{gather*}
	|E(N)| \les c_1 4^{\beta_1 + \ldots + \beta_r} \bigl( \prod_{j=1}^r \beta_j \bigr)  k e^{-k}. \\
	2) \mathop{{\sum}^{h}}_{f \in S_k ( \Gamma_0 (2) )} \prod_{j=1}^r {\lambda_f} (p_j^2)^{\beta_j} = F_2 (N) + E_2 (N),  
	\end{gather*} where
	\begin{gather*}
	F_2 (N) = \frac{k-1}{2\pi^2} \prod_{j=1}^r \sum_{k_j = 0}^{\beta_j} (-1)^{k_j} \binom{\beta_j}{k_j} \binom{k_j}{\lfloor k_j / 2 \rfloor}, \\
	|E_2 (N)| \les c_1 \bigl( \prod_{j=1}^r \beta_j^2 \bigr) 4^{\beta_1 + \ldots + \beta_r + r} k e^{-k} .
	\end{gather*}
\end{lemma} 

\begin{proof}
	
	First, we obtain Hecke relations
	$$
	\lambda_f (p)^n = \lambda_f (p^n) + \sum_{i=1}^{n/2} \biggl[ \binom{n}{i} - \binom{n}{i-1} \biggr] \lambda_f (p^{n - 2i}) 
	$$ if $n$ is even and
	$$
	\lambda_f (p)^n = \lambda_f (p^n) + \sum_{i=1}^{(n-1) / 2} \biggl[ \binom{n}{i} - \binom{n}{i-1} \biggr] \lambda_f (p^{n - 2i}) 
	$$ if $n$ is odd. From the Euler product
	$$
	\prod_p \biggl( 1 + \frac{\lambda_f (p)}{p^s} + \frac{\lambda_f (p^2)}{p^{2s}} + \ldots \biggr) = \prod_p \biggl( 1 + \frac{\alpha_p}{p^s} + \frac{\alpha_p^2}{p^{2s}} + \ldots \biggr) \biggl( 1 + \frac{\alpha_p^{-1}}{p^s} + \frac{\alpha_p^{-2}}{p^{2s}} + \ldots \biggr)
	$$ we get the identity
	\begin{equation} \label{identity_lambda}
	\lambda_f (p^{n-1}) = \frac{\alpha_p^n - \alpha_p^{-n}}{\alpha_p - \alpha_p^{-1}},
	\end{equation} and, in particular, $\lambda_f (p) = \alpha_p + \alpha_p^{-1}$. Next, we have
	\begin{multline*}
	\lambda_f (p)^n = (\alpha_p + {\alpha_p}^{-1})^n = \frac{(\alpha_p + {\alpha_p}^{-1})^n (\alpha_p - {\alpha_p}^{-1})}{\alpha_p - {\alpha_p}^{-1}} = \\ 
	\frac{\bigl(\alpha_p^n + n \alpha_p^{n-2} + \binom{n}{2} \alpha_p^{n-4} + \ldots + \binom{n}{i} \alpha_p^{n-2i} + \ldots + n \alpha_p^{2-n} + \alpha_p^{-n} \bigr) (\alpha_p - \alpha_p^{-1})}{\alpha_p - {\alpha_p}^{-1}}.
	\end{multline*} Opening the brackets in the numerator and applying~\eqref{identity_lambda} to each of the expressions
	$$
	\frac{\binom{n}{i} \alpha_p^{n-2i+1} - \binom{n}{i} \alpha_p^{-(n-2i+1)}}{\alpha_p - \alpha_p^{-1}},
	$$ we get the identity
	\begin{multline} \label{the_identity}
		\lambda_f (p)^n = \lambda_f (p^n) + [ n-1 ] \lambda_f (p^{n-2}) + \biggl[ \binom{n}{2} - n \biggr] \lambda_f (p^{n-4}) + \ldots \\
		+ \biggl[ \binom{n}{i} - \binom{n}{i-1} \biggr] \lambda_f (p^{n - 2i}) + \ldots .
	\end{multline} 
	
	First, assume $n$ is even. Then using~\eqref{the_identity} we obtain
	\begin{multline} \label{expansion_lambda}
	\mathop{{\sum}^{h}}_{f \in S_k ( \Gamma_0 (2) )} \lambda_f (p)^n [1 + \chi_{-d} (p)]^n = (1+\chi_{-d} (p))^n \mathop{{\sum}^{h}}_{f \in S_k ( \Gamma_0 (2) )} \lambda_f (p^n) + \\ 
	(1+\chi_{-d} (p))^n \mathop{{\sum}^{h}}_{f \in S_k ( \Gamma_0 (2) )} \sum_{i=1}^{n/2} \biggl[ \binom{n}{i} - \binom{n}{i-1} \biggr] \lambda_f (p^{n-2i}) = F(p^n) + E(p^n),
	\end{multline} where $F(p^n)$ is the term corresponding to $i = n/2$. By~Lemma~\ref{lem2},
	\begin{equation} \label{F_p_n}
	F(p^n) = (1+\chi_{-d} (p))^n \frac{k-1}{2\pi^2} \cdot \frac{2}{n+2} \binom{n}{n/2}.
	\end{equation} The error term can be bounded from above in the following way:
	\begin{multline} \label{E_p_n_even}
	E(p^n) \les 2^n \biggl( 1 + \sum_{i=1}^{n/2-1} \biggl[ \binom{n}{i} - \binom{n}{i-1} \biggr]  \biggr) c_1 k e^{-k} \les \\ 
	2^n \frac{n}{2} \binom{n}{n/2} c_1 k e^{-k} \les c_1 n 4^n k e^{-k}.
	\end{multline}
	
	Now assume that $n$ is odd. The main term in the Petersson's formula vanishes and for the error term by a similar argument we get
	\begin{equation} \label{E_p_n_odd}
	E(p^n) \les 2^n \frac{n-1}{2} \binom{n}{(n-1)/2} c_1 k e^{-k} \les c_1 n 4^n k e^{-k}.
	\end{equation} Combining~\eqref{expansion_lambda}, \eqref{F_p_n}, \eqref{E_p_n_even}, \eqref{E_p_n_odd}, Lemma~\ref{lem2} and using the multiplicativity of $\lambda_f$, we get the statement of part 1).

	The proof of the second part is similar. We have
	\begin{multline*}
	\lambda_f (p^2)^n = \frac{( \alpha_p^3 - \alpha_p^{-3})^n}{(\alpha_p - \alpha_p^{-1})^n} = \frac{(\alpha_p^2 + 1 + \alpha_p^{-2})^n (\alpha_p - \alpha_p^{-1})}{\alpha_p - \alpha_p^{-1}} = \\ \frac{\bigl( \sum_{l=0}^n \binom{n}{l} (\alpha_p^2 + \alpha_p^{-2})^l \bigr) (\alpha_p - \alpha_p^{-1})}{\alpha_p - \alpha_p^{-1}} = \frac{\sum_{l=0}^n \binom{n}{l} \sum_{i=0}^l \binom{l}{i} \alpha_p^{4i - 2l} (\alpha_p - \alpha_p^{-1})}{\alpha_p - \alpha_p^{-1}} = \\ \frac{\sum_{l=0}^n \binom{n}{l} \sum_{i=0}^l \binom{l}{i} [ \alpha_p^{4i - 2l + 1} - \alpha_p^{4i-2l-1} ]}{\alpha_p - \alpha_p^{-1}}.
	\end{multline*} The last expression can be written as
	\begin{multline} \label{big_expression}
	\sum_{\substack{l=1 \\ l \ \text{odd}}}^n \binom{n}{l} \biggl[ \sum_{i=0}^{(l-3)/2} \binom{l}{i} [\lambda_f (p^{2l-4i}) - \lambda_f (p^{2l-4i-2})] + \\ \binom{l}{(l-1)/2} [ \lambda_f (p^2) - \lambda_f (1) ] \biggr] + \\ \sum_{\substack{l=0 \\ l \ \text{even}}}^n \binom{n}{l} \biggl[ \sum_{i=0}^{(l-2)/2} \binom{l}{i} [\lambda_f (p^{2l-4i}) - \lambda_f (p^{2l-4i-2})] + \binom{l}{l/2} \lambda_f (1) \biggr].
	\end{multline} The main contribution to the sum over $f \in S_k (\Gamma_0 (2))$ after the application of~Lemma~\ref{lem2} would come from the terms corresponding to $\lambda_f (1)$:
	$$
	\mathop{{\sum}^{h}}_{f \in S_k ( \Gamma_0 (2) )} \sum_{l=0}^n \binom{n}{l} (-1)^l \binom{l}{\lfloor l/2 \rfloor} \lambda_f (1) = \frac{k-1}{2\pi^2} \sum_{l=0}^n (-1)^l \binom{n}{l} \binom{l}{\lfloor l/2 \rfloor} + E_{k,n},
	$$ where
	$$
	|E_{k,n}| \les c_1 k e^{-k} \sum_{l=0}^n \binom{n}{l} \binom{l}{\lfloor l/2 \rfloor} \les c_1 k e^{-k} (n+1) \binom{n}{\lfloor n/2 \rfloor} \binom{n}{\lfloor n/2 \rfloor} \les c_1 (n+1) 4^n k e^{-k}.
	$$ The contribution to the sum over $f \in S_k (\Gamma_0 (2))$ after the application of~Lemma~\ref{lem2} from the remaining terms in~\eqref{big_expression} would not exceed
	$$
	\sum_{l=0}^n \binom{n}{l} \sum_{i=0}^l 2 \binom{l}{i} c_1 k e^{-k} \les c_1 k e^{-k} \cdot 2 n^2 \binom{n}{\lfloor n/2 \rfloor} \binom{n}{\lfloor n/2 \rfloor} \les 2n^2 4^n c_1 k e^{-k}.
	$$ Finally, combining together these bounds and the identity~\eqref{big_expression}, we get 
	\begin{gather*}
	F_2 (p^n) = \frac{k-1}{2\pi^2} \sum_{l=0}^n (-1)^l \binom{n}{l} \binom{l}{\lfloor l/2 \rfloor}, \\
	|E_2 (p^n)| \les c_1 n^2 4^{n+1} k e^{-k}. 
	\end{gather*} The statement of part 2) follows from the multiplicativity of $\lambda_f$. 
	
\end{proof}

\begin{remark} In fact, we will only need a weaker form of the second part of~Lemma~\ref{lem3}, precisely, the bound
	\begin{equation} \label{crude_Prop2}
	|F_2 (N)| \les \frac{k-1}{2\pi^2} \prod_{j=1}^r \beta_j!,
	\end{equation} which follows from the inequality
	\begin{equation} \label{inequality_checked}
	S_{\beta} := \biggl| \sum_{k=0}^{\beta} \frac{(-1)^k}{(\beta - k)! (\lfloor k/2 \rfloor !)^2} \biggr| \les 1.
	\end{equation} Indeed, assume first that $\beta$ is odd and $\beta \ges 11$. Then
	\begin{multline*}
	S_{\beta} \les \biggl| \sum_{\substack{k=0 \\ k \text{ even}}}^{\beta-1} \biggl( \frac{1}{ (\beta-k)! (k/2)!^2} - \frac{1}{ (\beta - k - 1)! (k/2)!^2} \biggr) \biggr| \les
	\biggl| \sum_{\substack{k=0 \\ k \text{ even}}}^{\beta-1} \frac{k + 1 - \beta}{(\beta-k)! (k/2)!^2} \biggr| \les \\
	(\beta - 1) \sum_{\substack{k=0 \\ k \text{ even}}}^{\beta-1} \frac{1}{(\beta-k)! (k/2)!^2} \les (\beta-1) \sum_{\substack{k=0 \\ k \text{ even}}}^{\beta-1} \frac{1}{(\beta-k)! k!} \binom{k}{k/2} \les \\
	\frac{\beta - 1}{\beta !} \sum_{\substack{k=0 \\ k \text{ even}}}^{\beta-1} \binom{\beta}{k} \binom{k}{k/2} \les
	\frac{\beta-1}{\beta!} 2^{2\beta - 1} < 1
	\end{multline*} if $\beta \ges 11$. For even $\beta$ we similarly get
	$$
	S_{\beta} \les \frac{\beta-1}{\beta!} 2^{2\beta-1} + \frac{1}{(\beta/2)!^2} < 1
	$$ for $\beta > 11$. For the remaining values of $\beta$~\eqref{inequality_checked} can be verified directly: $S_0 = 1$, $S_1 = 0$, $S_2 = 1/2$, $S_3 = -1/3$, $S_4 = -3/8$, $S_5 = -11/30$, $S_6 = -7/48$, $S_7 = -33/280$, $S_8 = -181 / 5760$, $S_9 = -2113/45360$, $S_{10} = -2843/403200$. 
\end{remark}

\begin{lemma} \label{lem4}
	$1)$ Assume $X, Y, x$ are large real numbers such that $X^2 < Y$, $Y \les x$, and $M$ is a positive integer such that  $Y^M \les \sqrt{k}$, $\log \log Y \ges 5M$. Then if $M$ is even the following inequality holds true:
	\begin{multline} \label{first_ineq}
	\mathop{{\sum}^{h}}_{f \in S_k ( \Gamma_0 (2) )} \sum_{\substack{ p_1, \ldots, p_M \\ X < p_i \les Y}} \frac{\tilde \lambda_f (p_1) \ldots \tilde \lambda_f (p_M)}{\sqrt{p_1 \ldots p_M}} W (p_1, x) \ldots W (p_M, x) \les \\ 2 \frac{k-1}{2\pi^2} \frac{M!}{(M/2)!} \biggl( \sum_{X < p \les Y} \frac{1+\chi_{-d}(p)}{p} \biggr)^{M/2},
	\end{multline} where $\tilde \lambda_f (p) := \lambda_f (p) [1 + \chi_{-d} (p)]$, $p_1, \ldots, p_M$ are primes,
	$$
	W (p, x) = p^{-\lambda_0 / \log x} \frac{\log (x/p)}{ \log x}, \qquad \lambda_0 = 0.4912\ldots.
	$$ If $M$ is odd then
	\begin{multline} \label{second_ineq}
	\biggl| \mathop{{\sum}^{h}}_{f \in S_k ( \Gamma_0 (2) )} \sum_{\substack{ p_1, \ldots, p_M \\ X < p_i \les Y}} \frac{\tilde \lambda_f (p_1) \ldots \tilde \lambda_f (p_M)}{\sqrt{p_1 \ldots p_M}} W (p_1, x) \ldots W (p_M, x) \biggr| \les \\ \frac{(Y-X)^M}{X^{M/2}} \cdot 4^M \max_{\substack{p_1, \ldots, p_M \\ X < p_i \les Y}} g(p_1 \ldots p_M) \cdot c_1 k e^{-k},
	\end{multline} where for distinct primes $q_1, \ldots, q_r$ the function $g(.)$ is defined as 
	$$
	g(q_1^{\beta_1} \ldots q_r^{\beta_r}) = \prod_{i=1}^r \beta_i.
	$$
	
	$2)$ Let $m, N > 1$ be integers such that $N \les 2^m$, $(2^{m+1})^{2N} \les \sqrt{k}$, and $x$ be real number such that $2^{m+1} \les x$. Then
	\begin{multline*}
	\mathop{{\sum}^{h}}_{f \in S_k ( \Gamma_0 (2) )} \sum_{\substack{ p_1, \ldots, p_{2N} \\ 2^m < p_i \les 2^{m+1}}} \frac{\lambda_f (p_1^2) \ldots \lambda_f (p_{2N}^2)}{p_1 \ldots p_{2N}} W (p_1^2, x) \ldots W (p_{2N}^2, x) \les \\ 
	\sqrt{2 \pi} N 4^N \frac{k-1}{2\pi^2} \frac{(2N)!}{N!} \biggl( \frac{1}{2^m} \biggr)^N.
	\end{multline*}
\end{lemma}

\begin{proof}
	The inequality~\eqref{second_ineq} follows immediately from part 1) of Lemma~\ref{lem3} and the trivial upper bound 
	$$
	\sum_{X < p_1, \ldots, p_M \les Y} \frac{|W (p_1, x) \ldots W (p_M, x)|}{\sqrt{p_1 \ldots p_M}} \les \frac{(Y-X)^M}{X^{M/2}}.
	$$ 
	
	We now prove the inequality~\eqref{first_ineq}. Let $M := 2N$. We rewrite the product $p_1 \ldots p_{2N}$ in the canonical form $q_1^{\beta_1} \ldots q_r^{\beta_r}$, where all $q_i$ are distinct primes, all $\beta_i > 0$ and $\beta_1 + \ldots + \beta_r = 2N$. We split the sums over $p_1, \ldots, p_{2N}$ corresponding to the different patterns $(\beta_1, \ldots, \beta_r)$. Then we apply Lemma~\ref{lem3} to each of them. We have
	\begin{multline} \label{even_patterns}
	\mathop{{\sum}^{h}}_{f \in S_k ( \Gamma_0 (2) )} \sum_{X < p_1, \ldots, p_{2N} \les Y} \frac{\tilde \lambda_f (p_1) \ldots \tilde \lambda_f (p_{2N})}{\sqrt{p_1 \ldots p_{2N}}} W (p_1, x) \ldots W (p_{2N}, x) \les \\ 
	\mathop{{\sum}^{h}}_{f \in S_k ( \Gamma_0 (2) )} \sum_{1 \les r \les N} \sum_{\substack{\beta_1 + \ldots + \beta_r = 2N \\ \text{all even}}} \binom{2N}{\beta_1} \binom{2N - \beta_1}{\beta_2} \ldots \binom{2N - \beta_1 - \ldots - \beta_{r-1}}{\beta_r} \cdot \\ 
	\biggl( \frac{1}{r!} \sum_{\substack{X < q_1, \ldots, q_r \les Y \\ \text{all distinct}}} \frac{\tilde \lambda_f^{\beta_1} (q_1) \ldots \tilde \lambda_f^{\beta_r} (q_r)}{\bigl(q_1^{\beta_1} \ldots q_r^{\beta_r}\bigr)^{1/2}} \biggr) + R(X, Y, k),
	\end{multline} where we have removed the weights $W (p_i, x)$ from the main term since for each $i$ trivially $0 < W (p_i, x) \les 1$ and the product ${\tilde \lambda_f^{\beta_1}} (q_1) \ldots {\tilde \lambda_f^{\beta_r}} (q_r) \ges 0$ if all $\beta_i$ are even. The error term $R(X, Y, k)$ can be bounded similarly to~\eqref{second_ineq}:
	$$
	|R(X, Y, k)| \les 4^{2N} \max_{\substack{p_1, \ldots, p_{2N} \\ X < p_i \les Y}} g(p_1 \ldots p_{2N}) \frac{(Y-X)^{2N}}{X^N} \cdot c_1 k e^{-k}
	$$ by part 1) of Lemma~\ref{lem3}. The main contribution to the main term on the right hand side of~\eqref{even_patterns} comes from the pattern $\beta_1 = \ldots = \beta_r = 2$, $r = N$. By Lemma~\ref{lem3} it is equal to
	\begin{multline*}
	\frac{(2N)!}{2^N} \frac{1}{N!} \sum_{\substack{X < q_1, \ldots, q_N \les Y \\ \text{all distinct}}} \frac{1}{q_1 \ldots q_N} \biggl( \frac{k-1}{2\pi^2} \prod_{i=1}^N (1+\chi_{-d}(q_i))^2 + E(q_1^2 \ldots q_N^2) \biggr) \les \\ 
	\frac{k-1}{2\pi^2} \frac{(2N)!}{N!2^N} \biggl( \sum_{X < q \les Y} \frac{(1+\chi_{-d}(p))^2}{q} \biggr)^N + R_k \les \frac{k-1}{2\pi^2} \frac{(2N)!}{N!} \biggl( \sum_{X < q \les Y} \frac{1+\chi_{-d}(p)}{q} \biggr)^N + R_k,
	\end{multline*} where $R_k \les |R(X, Y, k)|$. 
	
	Next we apply Lemma~\ref{lem3} to the sum over the rest of the squares. Without loss of generality assume that $\beta_1, \ldots, \beta_s > 2$, $\beta_{s+1} = \ldots = \beta_r = 2$. The contribution from this pattern is
	\begin{multline} \label{pattern_contrib}
	\frac{(2N)!}{\beta_1 ! \ldots \beta_r ! r!} \biggl( \sum_{\substack{X < q_1, \ldots, q_r \les Y \\ \text{all distinct}}} \frac{\prod_{i=1}^r (1+\chi_{-d}(q_i))^{\beta_i}}{q_1^{\beta_1 / 2} \ldots q_s^{\beta_s/2} q_{s+1} \ldots q_r}  \biggr) \cdot \\
	\frac{k-1}{2\pi^2} \prod_{i=1}^r \frac{2}{\beta_i+2} \binom{\beta_i}{\beta_i/2} \bigl( 1 + o(1) \bigr).
	\end{multline} We estimate the contribution from all such patterns to the right hand side of~\eqref{even_patterns} in two different ways depending on the size of $X$. 
	
	First, assume $X$ is large. Precisely, $\log \log Y - \log \log X \les 8N$, so $\log \log X \ges 2N$. Then we apply crude estimates
	\begin{gather*}
	\sum_{X < q_i \les Y} \frac{1}{q_i^{\beta_i/2}} \les \frac{1}{(\beta_i/2)-1} \frac{1}{X^{\beta_i/2-1}},  \qquad i=1, \ldots, s, \\
	\sum_{X < q_i \les Y} \frac{1}{q_i} \les 9N, \qquad i = s+1, \ldots, r.
	\end{gather*} The contribution to the right hand side of~\eqref{even_patterns} does not exceed
	$$
	2^{2N} \frac{k-1}{2\pi^2} \frac{(2N)!}{\beta_1! \ldots \beta_r! r!} \biggl( \prod_{i=1}^s \frac{1}{(\beta_i / 2 - 1)} \frac{1}{X^{\beta_i/2 - 1}} \biggr) (9N)^{r-s} \biggl( \prod_{i=1}^r \frac{2}{\beta_i + 2} \binom{\beta_i}{\beta_i/2} \biggr) \bigl( 1 + o(1) \bigr),
	$$ which does not exceed
	\begin{multline*}
	2^{2N} \frac{k-1}{2\pi^2} \frac{(2N)!}{N! r!} \frac{N!}{(\beta_1/2)!^2 \ldots (\beta_r/2)!^2} \biggl( \prod_{i=1}^s \frac{1}{(\beta_i/2 - 1)} \frac{1}{X^{\beta_i/2 - 1}} \biggr) (9N)^{r-s} \les \\
	2^{2N} (9N)^N \frac{k-1}{2\pi^2} \frac{(2N)!}{N!} \frac{N!}{X \prod_{i=1}^r (\beta_i/2)!^2}.
	\end{multline*} Applying the inequalities
	$$
	\sum_{1 \les r \les N-1} \sum_{\substack{\beta_1 + \ldots + \beta_r = 2N \\ \text{all even}}} 1 \les 2^{N-1}, \qquad \frac{1}{\prod_{i=1}^r (\beta_i/2)!^2 } \les 1,
	$$ we bound the contribution from all such patterns by
	$$
	\frac{k-1}{2\pi^2} \frac{(2N)!}{N!} \frac{(9N)^{2N} N!}{X}.
	$$ This is negligible since
	$$
	\frac{(9N)^{2N} N!}{X} \qquad \text{is much smaller than} \qquad \biggl( \sum_{X < q \les Y} \frac{1+\chi_{-d}(q)}{q} \biggr)^N
	$$ due to the size of $X$.

	Now assume $\log \log X < 2N$. Trivially by $q^{-\beta_i/2} \les q^{-1}$ the expression~\eqref{pattern_contrib} can be bounded by
	$$
	2^{2N} \frac{k-1}{2\pi^2} \frac{(2N)!}{N! r!} \frac{N!}{\prod_{i=1}^r (\beta_i / 2)!^2 (1+\beta_i/2)} \biggl( \sum_{X < q \les Y} \frac{1}{q} \biggr)^{r-s}.
	$$ Then the result follows from
	$$
	\sum_{r \les N-1} \frac{1}{r!} \sum_{\gamma_1 + \ldots + \gamma_r = N} \frac{2^N N!}{\prod_{i=1}^r \gamma_i ! \gamma_i! (1+\gamma_i)} \biggl( \sum_{X < q \les Y} \frac{1}{q} \biggr)^{r-s-N} \ll 1,
	$$ since $\log \log Y - \log \log X \ges 8 N$.

	The proof of part 2) is similar. The main contribution comes from the squarefull products $p_1 \ldots p_{2N}$. We have
	\begin{multline*}
	\sum_{2^m < p_1, \ldots, p_{2N} \les 2^{m+1}} \frac{W (p_1^2, x) \ldots W (p_{2N}^2, x)}{p_1 \ldots p_{2N}} \mathop{{\sum}^{h}}_{f \in S_k ( \Gamma_0 (2) )} \lambda_f (p_1^2) \ldots \lambda_f (p_{2N}^2) = \\ 
	\sum_{\substack{2^m < p_1, \ldots, p_{2N} \les 2^{m+1} \\ \prod p_i \ \text{squarefull}}} \frac{W (p_1^2, x) \ldots W (p_{2N}^2, x)}{p_1 \ldots p_{2N}} \mathop{{\sum}^{h}}_{f \in S_k ( \Gamma_0 (2) )} \lambda_f (p_1^2) \ldots \lambda_f (p_{2N}^2) + S(m, N),
	\end{multline*} where $S(m, N)$ can be bounded by Lemma~\ref{lem3} part 2) from above as follows:
	\begin{multline*}
	|S(m, N)| \les 4^{2N+2N} \max_{p_1, \ldots, p_{2N}} h(p_1 \ldots p_{2N}) \biggl( \frac{2^{m+1} - 2^m}{2^m} \biggr)^{2N} \cdot c_1 k e^{-k} = \\ 4^{4N} \max_{p_1, \ldots, p_{2N}} h(p_1 \ldots p_{2N}) \cdot c_1 k e^{-k} =: R(m, N),
	\end{multline*} where
	$$
	h(p_1^{\alpha_1} \ldots p_r^{\alpha_r}) = \prod_{i=1}^r \alpha_i^2.
	$$ Note that $F_2 (p_1^{\beta_1} \ldots p_r^{\beta_r}) = 0$ if $\beta_i = 1$ for at least one $i$. We rewrite the main term in a way similar to the proof of part 1):
	\begin{multline} \label{this_expression}
	\sum_{1 \les r \les N} \sum_{\substack{\beta_1 + \ldots + \beta_r = 2N \\ \beta_i \ges  2}} \binom{2N}{\beta_1} \binom{2N-\beta_1}{\beta_2} \ldots \binom{2N - \beta_1 - \ldots - \beta_{r-1}}{\beta_r} \cdot \\ 
	\biggl( \frac{1}{r!} \sum_{\substack{2^m < q_1, \ldots, q_r \les 2^{m+1} \\ \text{all distinct}}} \frac{W (q_1^2, x)^{\beta_1} \ldots W (q_r^2, x)^{\beta_r}}{q_1^{\beta_1} \ldots q_r^{\beta_r}} \mathop{{\sum}^{h}}_{f \in S_k ( \Gamma_0 (2) )} \lambda_f^{\beta_1} (q_1^2) \ldots \lambda_f^{\beta_r} (q_r^2) \biggr).
	\end{multline} We bound this expression using the weaker form of~Lemma~\ref{lem3} part 2) (see~\eqref{crude_Prop2}). Thus,~\eqref{this_expression} does not exceed
	\begin{multline*}
	\frac{k-1}{2\pi^2} \sum_{1 \les r \les N} \sum_{\substack{\beta_1 + \ldots + \beta_r = 2N \\ \beta_i \ges  2}} \frac{(2N)!}{\beta_1 ! \ldots \beta_r! r!} \prod_{i=1}^r \beta_i ! \sum_{2^m < q_i \les 2^{m+1}} \frac{1}{q_i^{\beta_i}} + R(m, N) \les \\ 2 \frac{k-1}{2\pi^2} \sum_{1 \les r \les N} \sum_{\substack{\beta_1 + \ldots + \beta_r = 2N \\ \beta_i \ges  2}} \frac{(2N)!}{N!} \frac{N!}{r!} \prod_{i=1}^r \frac{1}{\beta_i - 1} \biggl( \biggl( \frac{1}{2^m} \biggr)^{\beta_i - 1} - \biggl( \frac{1}{2^{m+1}} \biggr)^{\beta_i - 1} \biggr) \les \\ 2 \frac{k-1}{2\pi^2} \frac{(2N)!}{N!} \biggl( \frac{1}{2^m} \biggr)^N \sum_{1 \les r \les N} \sum_{\substack{\beta_1 + \ldots + \beta_r = 2N \\ \beta_i \ges  2}}  \frac{N!}{r!} \biggl( \frac{1}{2^m} \biggr)^{N - r}.
	\end{multline*} By Stirling's formula
	$$
	\frac{N!}{r!} \biggl( \frac{1}{2^m} \biggr)^{N - r} \les \exp \biggl( N \log N - N + \log N + \frac{1}{2} \log 2 \pi - r \log r + r - (N - r) \log 2^m  \biggr).
	$$ The maximum of the last expression is achieved at $r = N$. In this case the only possible pattern is $(\beta_1, \ldots, \beta_N) = (2, \ldots, 2)$. Thus, it does not exceed $\sqrt{2 \pi} N$. Using the crude bound
	$$
	\sum_{1 \les r \les N} \sum_{\substack{\beta_1 + \ldots + \beta_r = 2N \\ \beta_i \ges  2}} 1 \les 2^{2N-1},
	$$ we get the desired result.
	
\end{proof}

Using the multiplicativity of $\lambda_f$ we can now combine both parts of~Lemma~\ref{lem4} together:

\begin{lemma} \label{lem5}
	Suppose $(X_0, Y_0], (X_1, Y_1], \ldots, (X_I, Y_I], (2^m, 2^{m+1}]$ are disjoint intervals, $X_i^2 < Y_i$. Let $M_1, \ldots, M_I$ be positive integers such that $\log \log Y_i \ges 5M_i$; $M_0$, $N_0$ be non-negative integers so that $M_0 \les 2^m$,
	$$
	(2^{m+1})^{M_0} B_0^{2N_0} \prod_{i=1}^I Y_i^{M_i} \les \sqrt{k},
	$$ and $b_0, b_1, \ldots$ be positive real numbers. Finally, let $x,y$ and $z$ be real numbers such that $x \ges \max( Y_1, \ldots, Y_I)$, $y \ges Y_0$, $z \ges 2^{m+1}$. Then the following two inequalities hold true:
	\begin{multline*}
	1) \mathop{{\sum}^{h}}_{f \in S_k ( \Gamma_0 (2) )} \prod_{i=1}^I \biggl( \sum_{m \les M_i} b_m \sum_{\substack{p_1, \ldots, p_m \\ X_i < p_j \les Y_i}} \frac{\tilde \lambda_f (p_1) \ldots \tilde \lambda_f (p_{ m})}{\sqrt{p_1 \ldots p_{m}}} W (p_1, x) \ldots W (p_m, x) \biggr) \cdot \\
	\biggl( \sum_{X_0 < p \les Y_0} \frac{\tilde \lambda_f (p)}{\sqrt{p}} W(p, y) \biggr)^{2N_0}
	\les \\
	3 \frac{k-1}{2 \pi^2} \biggl( \prod_{i=1}^I \sum_{n \les M_i / 2} b_{2n} \frac{(2n)!}{n!} \biggl( \sum_{X_i < p \les Y_i} \frac{1+\chi_{-d}(p)}{p} \biggr)^n \biggr) \cdot \frac{(2N_0)!}{N_0!} \biggl( \sum_{X_0 < p \les Y_0} \frac{1+\chi_{-d}(p)}{p} \biggr)^{N_0};
	\end{multline*}
	
	\begin{multline*}
	2) \mathop{{\sum}^{h}}_{f \in S_k ( \Gamma_0 (2) )} \biggl( \prod_{i=1}^I \sum_{m \les M_i} b_m \sum_{\substack{p_1, \ldots, p_m \\ X_i < p_j \les Y_i}} \frac{\tilde \lambda_f (p_1) \ldots \tilde \lambda_f (p_m)}{\sqrt{p_1 \ldots p_m}} W (p_1, x) \ldots W (p_m, x) \biggr) \cdot \\
	\biggl( \sum_{X_0 < p \les Y_0} \frac{\tilde \lambda_f (p)}{\sqrt{p}} W(p, y) \biggr)^{2N_0} \cdot \biggl( \sum_{2^m < p \les 2^{m+1}} \frac{\lambda_f (p^2)}{p} W (p^2, z) \biggr)^{2M_0} \les \\ 
	3 \frac{k-1}{2 \pi^2}  \biggl( \prod_{i=1}^I \sum_{n \les M_i / 2} b_{2n} \frac{(2n)!}{n!} \biggl( \sum_{X_i < p \les Y_i} \frac{1+\chi_{-d}(p)}{p} \biggr)^n \biggr) \cdot \frac{(2N_0)!}{N_0!} \biggl( \sum_{X_0 < p \les Y_0} \frac{1+\chi_{-d}(p)}{p} \biggr)^{N_0} \cdot \\
	\sqrt{2 \pi} M_0 4^{M_0} \frac{(2M_0)!}{M_0!} \biggl( \frac{1}{2^m} \biggr)^{M_0}. 
	\end{multline*}
	
\end{lemma}

\begin{lemma} \label{lem6}
	Let $x$ be a large real number, and $d$ be a large integer. Assuming GRH for $L(s, \chi_{-d})$ one has
	$$
	L(1, \chi_{-d}) \ges \exp \biggl( \sum_{p \les x} \frac{\chi_{-d}(p)}{p} (1 - o(1)) \biggr).
	$$
\end{lemma} This lemma follows from~\cite[Lemma~3.1]{Lumley}.

\section{Moment Estimate. Proof of Theorem 2}
\label{sec6}

We first set the notation which will be used throughout the proof. It is similar to the one used in Harper's work. Then we give a brief sketch of the proof and finish the section by a detailed computation.

\subsection{Notation}
\label{subsec6.1}

Let $k$ be a weight of a holomorphic cusp form $f$ and $d$ be a modulus of a Dirichlet character $\chi_{-d}$. We fix the numbers $x, D$ so that $x = k^{1/D}$. The parameter $D$ will be chosen later. Next, we introduce the notation for Dirichlet polynomials in the bound from Lemma~\ref{lem1}:
\begin{gather*}
G_1 (f, y) := \sum_{2 < p \les X_1} \frac{\tilde \lambda_f (p)}{\sqrt{p}} W(p, y), \qquad G_i (f, y) := \sum_{X_{i-1} < p \les X_i} \frac{\tilde \lambda_f (p)}{\sqrt{p}} W(p, y),  \\
X_i := x^{3^{i-I}}, \qquad i = 1, \ldots, I, \\
{\tilde \lambda_f (p)} := \lambda_f (p) [1 + \chi_{-d} (p)], \qquad I = \lfloor \log \log \log x \rfloor. 
\end{gather*} We split the set $f \in S_k (\Gamma_0 (2))$ into the union of the subsets $\EuScript{F} \cup S(0) \cup S(1) \cup \ldots \cup S(I-1)$, where
\begin{gather*}
\EuScript{F} := \biggl\{ f \in S_k ( \Gamma_0 (2) ): |G_i (f, x)| \les B_i, \ i = 1, \ldots, I  \biggr\}, \\
S(j) := \biggl\{ f \in S_k ( \Gamma_0 (2) ): \bigl|  G_i (f, y) \bigr| \les B_i, \ i = 1, \ldots, j, \text{ for all } y \ges X_i; \\
| G_{j+1} (f, y_0) | > B_{j+1} \text{ for some } y_0 \ges X_{j+1} \biggr\}, \\
S(0) := \biggl\{ f \in S_k ( \Gamma_0 (2) ): | G_1 (f, y_0) | > B_1 \text{ for some } y_0 \ges X_1 \biggr\}.
\end{gather*} The numbers $B_i$ will be chosen later. Similarly we introduce the notation for the second term polynomials from Lemma~\ref{lem1}. First, let
$$
P_m (f, y) := \sum_{2^m < p \les 2^{m+1}} \frac{\lambda_f (p^2)}{p} W (p^2, y).
$$ Next, for $m > 0$ define the sets
\begin{gather*}
\EuScript{P}(m) := \biggl\{ f \in S_k ( \Gamma_0 (2) ):  P_m (f, y_0) > m^{-2} \text{ for some } y_0 \ges 2^{m+1}; \\
P_n (f, y) \les n^{-2} \text{ for all } y \ges 2^{m+1} \text{ and for all } n \text{ such that } m < n \les \frac{\log x}{2 \log 2}  \biggr\}.
\end{gather*} Furthermore, we define
\begin{gather*}
\EuScript{P}(0) := \biggl\{ f \in S_k ( \Gamma_0 (2) ):  P_0 (f, y_0) > 1 \text{ for some } y_0 \ges 2; \\
P_n (f, y) \les n^{-2} \text{ for all } y \ges 2 \text{ and for all } n \text{ such that } 0 < n \les \frac{\log x}{2 \log 2}  \biggr\},
\end{gather*} and
\begin{gather*}
\EuScript{P} := \biggl\{ f \in S_k ( \Gamma_0 (2) ): P_0 (f, y) \les 1 \text{ for all } y \ges 2; \\
P_n (f, y) \les n^{-2} \text{ for all } y \ges 2 \text{ and for all } n \text{ such that } 0 < n \les \frac{\log x}{2 \log 2} \biggr\}.
\end{gather*}

\subsection{Sketch of the proof}
\label{subsec6.2}

We bound the original sum as follows:
\begin{equation} \label{original_sum}
\mathop{{\sum}^{h}}_{f \in S_k ( \Gamma_0 (2) )} \exp \bigl( \log \bigl| L_1 (f) L_2 (f) \bigr| \bigr) \les \mathop{{\sum}^{h}}_{f \in \EuScript{F}} \exp \bigl(\ldots \bigr) + \sum_{j=0}^{I-1} \mathop{{\sum}^{h}}_{f \in S(j)} \exp \bigl(\ldots \bigr).
\end{equation} Then we additionally split each of the sums in the right hand side of the last inequality using the subsets $\EuScript{P}(m)$:
\begin{gather*}
\mathop{{\sum}^{h}}_{f \in \EuScript{F}} \exp \bigl(\ldots \bigr) \les \biggl( \mathop{{\sum}^{h}}_{f \in \EuScript{F} \cap \EuScript{P}} + \sum_{0 \les m \les \log x / (2\log 2)} \mathop{{\sum}^{h}}_{f \in \EuScript{F} \cap \EuScript{P} (m)} \biggr) \exp \bigl(\ldots \bigr), \\	 
\mathop{{\sum}^{h}}_{f \in S(j)} \exp \bigl(\ldots \bigr) \les \biggl( \mathop{{\sum}^{h}}_{f \in S(j) \cap \EuScript{P}} + \sum_{0 \les m \les \log x / (2\log 2) } \mathop{{\sum}^{h}}_{f \in S(j) \cap \EuScript{P} (m)} \biggr) \exp \bigl(\ldots \bigr)
\end{gather*} for all $j = 1, \ldots, I-1$. Thus, the goal is to obtain the upper bound for
\begin{multline*}
\biggl( \sum_{0 \les m \les \log x / (2\log 2)}  \mathop{{\sum}^{h}}_{f \in \EuScript{F} \cap \EuScript{P}(m)} + \\
\sum_{j=1}^{I-1} \sum_{0 \les m \les \log X_j / (2\log 2)} \mathop{{\sum}^{h}}_{f \in S(j) \cap \EuScript{P}(m)} + \mathop{{\sum}^{h}}_{f \in S(0)} \biggr) \exp \bigl( \log \bigl|L_1 (f) L_2 (f) \bigr| \bigr),
\end{multline*} and for
$$
\biggl(  \mathop{{\sum}^{h}}_{f \in \EuScript{F} \cap \EuScript{P}} + \sum_{j=1}^{I-1} \mathop{{\sum}^{h}}_{f \in S(j) \cap \EuScript{P}} \biggr) \exp \bigl( \log \bigl|L_1 (f) L_2 (f) \bigr| \bigr).
$$

By Lemma~\ref{lem1} the problem reduces to bounding of a product of two Dirichlet polynomials
\begin{equation} \label{two_factors}
\exp \biggl( \sum_{p \les x} \frac{\tilde \lambda_f (p)}{\sqrt{p}} W (p, x) \biggr) \exp \biggl( \sum_{p \les \sqrt{x}} \frac{\lambda_f (p^2)}{p} W (p^2, x) \biggr).
\end{equation} The first one is split to the intervals $( X_{i-1}, X_i]$ defined earlier, and the second one is split to the intervals of the form $( 2^m , 2^{m+1} ]$. 

In the first step, we ignore the second factor in~\eqref{two_factors} and estimate the sum
$$
\mathop{{\sum}^{h}}_{f \in \EuScript{F}} \exp \biggl( \sum_{p \les x} \frac{\tilde \lambda_f (p)}{\sqrt{p}} W (p, x) \biggr) = \mathop{{\sum}^{h}}_{f \in \EuScript{F}} \prod_{i=1}^I \exp \bigl( G_i (f, x) \bigr).
$$ Using Taylor expansion, we get
$$
\prod_{i=1}^I \exp \bigl( G_i (f, x) \bigr) = \prod_{i=1}^I \biggl( \sum_{l \les A_i} \frac{1}{l!} (G_i (f, x))^l + r (A_i, f) \biggr),
$$ where the error $r(A_i, f)$ is small by definition of the set $\EuScript{F}$. Next, we expand the $l$-powers of $G_i(f, x)$ and the product over $1 \les i \les I$, and then apply the first part of Lemma~\ref{lem5} to the multiple sums over primes we obtain after the expansion. 

To treat the sum over the exceptional sets $S(j)$ we use a variation of Rankin's trick:
$$
\mathop{{\sum}^{h}}_{f \in S(j) } \exp \bigl( |L_1 (f) L_2 (f)| \bigr) \les \mathop{{\sum}^{h}}_{f \in S_k (\Gamma_0 (2)) } \exp \bigl( L_1 (f) L_2 (f) \bigr) \biggl( \frac{G_{j+1} (f, y_0)}{B_{j+1}} \biggr)^{2N}.
$$ Here $y_0 = y_0 (j)$. We gain from the fact that the last factor is $\ll 1$ for most $f$ in $S_k (\Gamma_0 (2))$. The choice of $N$ is bounded by the restrictions of Lemmas~\ref{lem4} and~\ref{lem5}. The smaller values of $j$ correspond to the shorter Dirichlet polynomials and produce larger error terms in the Taylor expansion, so they correspond to the larger choices of $A_{j+1}$ and $B_{j+1}$. For the exceptional set $S(0)$ the Rankin's trick is no longer enough, so we combine Cauchy and Markov inequalities to get the additional saving from the small size of the set~$S(0)$. 

The same strategy works for the second-term polynomials $P_m (f, y)$. For the sets $\EuScript{P}$ and $\EuScript{P}(m)$ for $m$ bounded by a large constant $C_1$, the polynomial $P_m (f, y)$ can be trivially bounded by another constant. Next, if $C_1 < m \les 2 \log \log x$, we use a crude bound
$$
\sum_{p \les 2^m} \frac{\tilde \lambda_f (p)}{\sqrt{p}} W (p, x) + \sum_{p \les 2^m} \frac{\lambda_f (p^2)}{p} W (p^2, x) \les 20 \frac{2^{m/2}}{m}.
$$ We compensate the contribution from this factor using Rankin's trick:
\begin{multline*}
\mathop{{\sum}^{h}}_{f \in \EuScript{F} \cap \EuScript{P}(m)} \exp \biggl( \sum_{p \les x} \frac{\tilde \lambda_f (p)}{\sqrt{p}} W (p, x) + \sum_{p \les \sqrt{x}} \frac{\lambda_f (p^2)}{p} W (p^2, x) \biggr) \les \\ 
\exp \biggl( 20 \frac{2^{m/2}}{m} \biggr) \mathop{{\sum}^{h}}_{f \in \EuScript{F}} \exp \biggl( \sum_{2^m < p \les x} \frac{\tilde \lambda_f (p)}{\sqrt{p}} W (p, x) \biggr) \biggl( \frac{P_m (f, y_0)}{1 / m^2} \biggr)^{2M}.
\end{multline*} Here $y_0 = y_0 (m)$. The last factor is small for most $f$ in $\EuScript{F}$. The choice of $M$ is bounded by the restrictions of Lemmas~\ref{lem4} and~\ref{lem5}. For $m > 2 \log \log x$ Rankin's trick is no longer enough, so we apply Cauchy inequality to gain from the small sizes of the sets $\EuScript{P} (m)$.

\subsection{The estimation for $f \in \EuScript{F}$}
\label{subsec6.3}

In this section, we only consider the contribution from the first Dirichlet polynomials $G_i (f, y)$. Precisely, we prove the inequality
\begin{equation} \label{estimate_for_F}
\mathop{{\sum}^{h}}_{f \in \EuScript{F}} \exp \biggl( \sum_{p \les x} \frac{\tilde \lambda_f (p)}{\sqrt{p}} W (p, x) \biggr) \les k L(1, \chi_{-d}) \bigl( \log x \bigl),	
\end{equation} which gives the desired bound in the case $f \in \EuScript{F} \cap \EuScript{P}$ and $f \in \EuScript{F} \cap \EuScript{P}(m)$ for $m \ll 1$. It follows from the fact that the contribution from all the second polynomials $P_m (f, y)$ is trivially bounded by a constant.  

In accordance with the notation we split the sum over $p \les x$ as follows:
$$
\mathop{{\sum}^{h}}_{f \in \EuScript{F}} \exp \biggl( \sum_{p \les x} \frac{\tilde \lambda_f (p)}{\sqrt{p}} W (p, x) \biggr) = \mathop{{\sum}^{h}}_{f \in \EuScript{F}} \prod_{i = 1}^I \exp \biggl(\frac{1}{2}  \sum_{X_{i-1} < p \les X_i} \frac{\tilde \lambda_f (p) }{\sqrt{p}} W (p, x) \biggr)^2.
$$ Next, we expand each factor $\exp(\ldots)$ into a Taylor series using Lagrange error term. Thus, the last expression does not exceed
$$
\mathop{{\sum}^{h}}_{f \in \EuScript{F}} \prod_{i=1}^I \biggl[ \sum_{l \les A_i} \frac{1}{l!} \biggl( \frac{1}{2} G_i (f, x) \biggr)^l + \frac{\exp \bigl( (1/2) |G_i (f, x)| \bigr)}{(A_i + 1)!} \biggl( \frac{1}{2}  \ |G_i (f, x)| \biggr)^{A_i + 1} \biggr]^2,
$$ Choose $A_i = 10 B_i$, so that the error term is sufficiently small. We rewrite the last expression as
\begin{multline} \label{lpowers}
\mathop{{\sum}^{h}}_{f \in \EuScript{F}} \prod_{i=1}^I (1 + \varepsilon_i)^2 \biggl( \sum_{l \les A_i} \frac{1}{l!} \biggl( \frac{1}{2}  \ G_i (f, x) \biggr)^l \biggr)^2 \les \\
(1 + \varepsilon_0) \mathop{{\sum}^{h}}_{f \in \EuScript{F}} \prod_{i=1}^I \biggl( \sum_{l \les A_i} \frac{1}{l!} \biggl( \frac{1}{2} G_i (f, x) \biggr)^l \biggr)^2,
\end{multline} where
\begin{multline*}
\varepsilon_i := \biggl( \sum_{l \les A_i} \frac{1}{l!} \biggl( \frac{1}{2} |G_i (f, x)| \biggr)^l \biggr)^{-1} \frac{\exp \bigl( (1/2) |G_i (f, x)| \bigr)}{(A_i + 1)!} \biggl( \frac{1}{2} |G_i (f, x)| \biggr)^{A_i+1} \les \\
\frac{\exp \bigl( B_i / 2 \bigr)}{(A_i + 1)!} \biggl( \frac{B_i}{2} \biggr)^{A_i + 1},
\end{multline*} and
\begin{multline*}
\prod_{i=1}^I (1 + \varepsilon_i)^2 \les \exp \biggl(  2 \sum_{i=1}^I \varepsilon_i \biggr) = \\
\exp \biggl( 2 \sum_{i=1}^I \exp \biggl( \frac{B_i}{2} - \log (A_i + 1)! + (A_i + 1) \log \frac{B_i}{2} \biggr) \biggr) =: 1 + \varepsilon_0.
\end{multline*} By Stirling formula, the last expression can further be bounded as
\begin{multline*}
\exp \biggl( 2 \sum_{i=1}^I \exp \biggl( \frac{B_i}{2} - (A_i + 1) \log (A_i + 1) + A_i + 1 + (A_i + 1) \log \frac{B_i}{2} \biggr)  \biggr) \les \\
\exp \biggl( 2 \sum_{i=1}^I \exp ( -A_i ) \biggr).
\end{multline*}

The right hand side of~\eqref{lpowers} trivially does not exceed the sum over the full basis $S_k (\Gamma_0 (2))$
$$
(1 + \varepsilon_0) \mathop{{\sum}^{h}}_{f \in S_k ( \Gamma_0 (2) )} \prod_{i=1}^I \biggl( \sum_{l \les A_i} \frac{1}{l!} \biggl( \frac{1}{2} G_i (f, x) \biggr)^l \biggr)^2.
$$ Next, we expand the square and $l$-powers in the right in the last expression:
$$
(1 + \varepsilon_0) \mathop{{\sum}^{h}}_{f \in S_k ( \Gamma_0 (2) )} \sum_{\tilde l, \tilde t}  \biggl(\prod_{i=1}^I \frac{1}{l_i! t_i! 2^{l_i + t_i}} \biggr)  \sum_{\tilde p, \tilde q} \biggl( \prod_{\substack{i \les I \\ r \les l_i \\ s \les t_i}} \frac{\tilde \lambda_f (p_{i,r}) \tilde \lambda_f (q_{i,s})}{\sqrt{p_{i,r} q_{i,s}}} W(p_{i,r}, x) W(q_{i,s}, x) \biggr),
$$ where the sum in the middle runs over all the vectors of the form $\tilde l = (l_1, \ldots, l_I)$, $\tilde t = (t_1, \ldots, t_I)$, such that $0 \les l_i, t_i \les A_i$, and the inner sum runs over all the vectors of primes
\begin{gather*}
\tilde p (\tilde l) = (p_{1,1}, \ldots, p_{1,l_1}, p_{2,1}, \ldots, p_{2,l_2}, \ldots, p_{I,1}, \ldots, p_{I,l_I}), \\
\tilde q (\tilde t) = (q_{1,1}, \ldots, q_{1,t_1}, q_{2,1}, \ldots, q_{2,t_2}, \ldots, q_{I,1}, \ldots, q_{I,t_I}), \\
p_{i, r}, q_{i,s} \in ( X_{i-1}, X_i].
\end{gather*} We rewrite the last expression in the following way:
\begin{multline*}
(1 + \varepsilon_0) \mathop{{\sum}^{h}}_{f \in S_k ( \Gamma_0 (2) )} \prod_{i=1}^I \biggl( \sum_{l,t \les A_i} \frac{1}{l! t! 2^{l+t}} \\
\sum_{\substack{p_1, \ldots, p_l \\ q_1, \ldots, q_t \\ \in ( X_{i-1}, X_i ]}} \frac{\tilde \lambda_f (p_1) \ldots \tilde \lambda_f (p_l) \tilde \lambda_f (q_1) \ldots \tilde \lambda_f (q_t)}{\sqrt{p_1 \ldots p_l q_1 \ldots q_t}} W(p_1, x)\ldots W(p_l, x) W(q_1, x) \ldots W(q_t, x) \biggr),
\end{multline*} which is equivalent to
\begin{multline*}
(1 + \varepsilon_0) \mathop{{\sum}^{h}}_{f \in S_k ( \Gamma_0 (2) )} \prod_{i=1}^I \biggl( \sum_{m \les 2A_i} \frac{1}{2^m} \biggl( \sum_{l+t = m} \frac{1}{l! t!} \biggr) \\ \sum_{X_{i-1} < p_1, \ldots, p_m \les X_i} \frac{\tilde \lambda_f (p_1) \ldots \tilde \lambda_f (p_m)}{\sqrt{p_1 \ldots p_m}} W(p_1, x) \ldots W(p_m, x) \biggr),
\end{multline*} which is
\begin{multline} \label{after_binomial}	
(1 + \varepsilon_0) \mathop{{\sum}^{h}}_{f \in S_k ( \Gamma_0 (2) )} \prod_{i=1}^I \biggl( \sum_{m \les 2 A_i} \frac{1}{m!} \\
\sum_{X_{i-1} < p_1, \ldots, p_m \les X_i} \frac{\tilde \lambda_f (p_1) \ldots \tilde \lambda_f (p_m)}{\sqrt{p_1 \ldots p_m}} W(p_1, x) \ldots W(p_m, x) \biggr). 
\end{multline} Note that the intervals $(X_{i-1}, X_i]$ are disjoint for different values of $i \les I$. Together with inequality
$$
\prod_{i=1}^I X_i^{2A_i} \les \sqrt{k}
$$ it implies that the assumptions of Lemma~\ref{lem5} hold. Take $b_m = (m!)^{-1}$, $N_0 = 1$. By the first part of Lemma~\ref{lem5} the expression~\eqref{after_binomial} is bounded by
$$
(1 + \varepsilon_0) \cdot 3 \frac{k-1}{2\pi^2} \prod_{i=1}^I \biggl( \sum_{n \les A_i} \frac{1}{(2n)!} \frac{(2n)!}{n!} \biggl( \sum_{X_{i-1} < p \les X_i} \frac{1+\chi_{-d}(p)}{p} \biggr)^n \biggr), 
$$ which by Lemma~\ref{lem6} can further be bounded as
$$
(1 + \varepsilon_0) \cdot 3\frac{k-1}{2\pi^2} \exp \biggl( \sum_{p \les x} \frac{1 + \chi_{-d} (p)}{p} \biggr) \les k L(1, \chi_{-d}) \bigl(\log x \bigr).
$$

\subsection{The estimation for $f \in S(j)$}
\label{subsec6.4}

In this section, we show that the contribution from each exceptional set $S(j)$ to the left hand side of~\eqref{original_sum} is small. Precisely, we prove the inequality
\begin{multline} \label{goal_section_sj}
\mathop{{\sum}^{h}}_{f \in S(j)} \exp \biggl( \sum_{p \les X_j} \frac{\tilde \lambda_f (p)}{\sqrt{p}} W(p, X_j) - \sum_{p \les X_j^{1/2}} \frac{1}{p} W(p^2, X_j) + c_0 + \frac{3}{4} \frac{\log d^2 k^4}{\log X_j}  \biggr) \ll  \\
k L(1, \chi_{-d}) \exp \biggl(  -\frac{3^{I-j-1} D}{32} \log \frac{D}{8} \biggr)
\end{multline} for each $1 \les j \les I-1$. Here $D$ is given by the identity $x = k^{1/D}$. Note that we have removed the sum
$$
\sum_{p \les X_j^{1/2}} \frac{\lambda_f (p^2)}{p} W(p^2, X_j)
$$ from the exponent in the left hand side of~\eqref{goal_section_sj}. We will take care of the contribution from this sum in the next part of this section. Next, by the definition of the sets $S(j)$, we have
\begin{multline*}
\mathop{{\sum}^{h}}_{f \in S(j)} \exp \biggl( \sum_{p \les X_j} \frac{\tilde \lambda_f (p)}{\sqrt{p}} W(p, X_j)  \biggr) = \mathop{{\sum}^{h}}_{f \in S(j) } \prod_{i=1}^j \exp \biggl(  \frac{1}{2} \sum_{X_{i-1} < p \les X_i} \frac{\tilde \lambda_f (p)}{\sqrt{p}} W(p, X_j) \biggr)^2 \les \\ 
\mathop{{\sum}^{h}}_{f \in S(j) } \prod_{i=1}^j \exp \biggl( \frac{1}{2} G_i (f, X_j)  \biggr)^2 \biggl( \frac{G_{j+1} (f, y_0)}{B_{j+1}} \biggr)^{2 N_{j+1}}
\end{multline*} for any integer $N_{j+1} > 0$. Similarly to the case $f \in \EuScript{F}$ we expand $\exp((1/2) G_i (f, X_j))$ into a Taylor series and bound the right hand side of the last inequality by
\begin{equation} \label{expression_first_form}
(1 + \varepsilon_0^{(j)}) {B_{j+1}}^{-2N_{j+1}} \mathop{{\sum}^{h}}_{f \in S_k ( \Gamma_0 (2) ) } \bigl( G_{j+1} (f, y_0) \bigr)^{2N_{j+1}} \prod_{i=1}^j \biggl( \sum_{l \les A_i} \frac{1}{l!} \biggl( \frac{1}{2} G_i (f, X_j) \biggr)^l \biggr)^2,
\end{equation} where
$$
1 + \varepsilon_0^{(j)} := \exp \biggl( 2 \sum_{i=1}^j \exp \biggl( \frac{B_i}{2} - \log (A_i + 1)! + (A_i + 1) \log \frac{B_i}{2} \biggr) \biggr). 
$$ Expanding out $2N_{j+1}$- and $l$-powers of polynomials $G_i (f, .)$, we write the expression~\eqref{expression_first_form} as
\begin{multline*}
(1 + \varepsilon_0^{(j)}) \bigl( {B_{j+1}}^{-2N_{j+1}} \bigr) \mathop{{\sum}^{h}}_{f \in S_k ( \Gamma_0 (2) ) } \biggl( \prod_{i=1}^j \sum_{m \les 2 A_i} \frac{1}{m!} \sum_{X_{i-1} < p_1, \ldots, p_m \les X_i} \frac{\tilde \lambda_f (p_1) \ldots \tilde \lambda_f (p_m)}{\sqrt{p_1 \ldots p_m}} \cdot \\
W(p_1, X_j) \ldots W(p_m, X_j) \biggr) \cdot \\
\biggl( \sum_{X_j < p_1, \ldots, p_{2N_{j+1}} \les X_{j+1}} \frac{\tilde \lambda_f (p_1) \ldots \tilde \lambda_f (p_{2N_{j+1}})}{\sqrt{p_1 \ldots p_{2N_{j+1}}}} W(p_1, y_0) \ldots W(p_{2N_{j+1}}, y_0) \biggr).
\end{multline*} By the first part of~Lemma~\ref{lem5}, this does not exceed
\begin{multline*}
(1 + \varepsilon_0^{(j)}) B_{j+1}^{-2N_{j+1}} \cdot 3 \frac{k-1}{2\pi^2}  
\biggl( \prod_{i=1}^j \sum_{n \les A_i} \frac{1}{n!} \biggl( \sum_{X_{i-1} < p \les X_i} \frac{1+\chi_{-d} (p)}{p} \biggr)^n \biggr) \\
\cdot \frac{(2N_{j+1})!}{N_{j+1} !} \biggl( \sum_{ X_j < p \les X_{j+1} } \frac{1+\chi_{-d}(p)}{p} \biggr)^{N_{j+1}}. 
\end{multline*} Then using the bounds
\begin{gather*}
\prod_{i=1}^j \sum_{n \les A_i} \frac{1}{n!} \biggl( \sum_{X_{i-1} < p \les X_i} \frac{1+\chi_{-d}(p)}{p} \biggr)^n \les \exp \biggl( \sum_{p \les X_j} \frac{1 + \chi_{-d} (p)}{p} \biggr), \\
\biggl( \sum_{ X_j < p \les X_{j+1}} \frac{1+\chi_{-d}(p)}{p} \biggr)^{N_{j+1}} \les (\log 3 + 1)^{N_{j+1}}
\end{gather*} and Stirling formula, we get
\begin{multline*}
\mathop{{\sum}^{h}}_{f \in S(j)} \exp \biggl( \sum_{p \les X_j} \frac{\tilde \lambda_f (p)}{\sqrt{p}} W(p, X_j) \biggl) \les \\
(1 + \varepsilon_0^{(j)}) \cdot 3 \frac{k-1}{2\pi^2} \exp \biggl( -2N_{j+1} \log B_{j+1} + 2N_{j+1} \log 2N_{j+1} - 2 N_{j+1} + \log (2N_{j+1}) + \\ 
\frac{1}{2} \log 2 \pi - N_{j+1} \log N_{j+1} + N_{j+1} + \sum_{p \les X_j} \frac{1 + \chi_{-d} (p)}{p} \biggr) \bigl( \log 3 + 1 \bigr)^{N_{j+1}}.   
\end{multline*} By Lemma~\ref{lem6}, the right hand side of the last inequality is bounded by
\begin{multline} \label{after_first_Stirling}
k L(1, \chi_{-d}) \exp \biggl( -2N_{j+1} \log B_{j+1} + N_{j+1} \log N_{j+1} + \\
N_{j+1} \bigl[ 2 \log 2 - 1 + \log (\log 3 + 1)\bigr] + \log N_{j+1} + \log \log X_j + c_2 \biggr) 
\end{multline} with an absolute constant $c_2 > 0$. Following the restrictions of~Lemma~\ref{lem5}
$$
\bigl( \prod_{i=1}^j X_i^{2A_i} \bigr) X_{j+1}^{2N_{j+1}} \les \sqrt{k},
$$ we choose
\begin{gather*}
\log (X_1)^{2A_1} = \frac{1}{3 \sqrt{3}} \biggl( \frac{e}{3} \biggr)^{I-1} \log \sqrt{k}, \\ 
\log (X_I)^{2A_I} = \frac{1}{9} \log \sqrt{k}, \qquad \log (X_{I-1})^{2A_{I-1}} = \frac{1}{9 \sqrt{3}} \log \sqrt{k}, \qquad \ldots, \\
\log (X_j)^{2A_j} = \frac{1}{(\sqrt{3})^{I-j+4}} \log \sqrt{k}, \qquad \ldots, \qquad \log (X_2)^{2A_2} = \frac{1}{(\sqrt{3})^{I+2}} \log \sqrt{k}. 
\end{gather*} Then
\begin{gather*}
A_1 = \frac{1}{12 \sqrt{3}} \biggl( \frac{e}{3} \biggr)^{I-1} \frac{\log k}{\log X_1} = \frac{De^{I-1}}{12 \sqrt{3}}, \\
A_I = \frac{D}{36}, \qquad \ldots, \qquad A_j = \frac{D}{36} \bigl( \sqrt{3} \bigr)^{I-j}, \qquad \ldots, \qquad A_2 = \frac{D}{36} \bigl( \sqrt{3} \bigr)^{I-2}.
\end{gather*} Further, choose
$$
N_{j+1} = \biggl\lfloor \frac{D}{8} 3^{I-j-1} \biggr\rfloor. 
$$ Thus, the expression~\eqref{after_first_Stirling} is bounded by
\begin{multline*}
k L(1, \chi_{-d}) \exp \biggl( -2 \biggl\lfloor \frac{3^{I-j-1} D}{8} \biggr\rfloor \log \frac{\bigl( \sqrt{3} \bigr)^{I-j-1} D}{360} + \biggl\lfloor \frac{3^{I-j-1} D}{8} \biggr\rfloor \log \biggl\lfloor \frac{3^{I-j-1} D}{8} \biggr\rfloor + \\ 
2 \biggl\lfloor \frac{3^{I-j-1} D}{8} \biggr\rfloor + \log \log X_j + c_2 \biggr) \les k L(1, \chi_{-d}) \exp \biggl( -\frac{3^{I-j-1} D}{16} \log \frac{D}{8} + \log \log X_j + c_2 \biggr)
\end{multline*} provided that $D > 10^{14}$. Finally, we get
\begin{multline*}
\mathop{{\sum}^{h}}_{f \in S(j)} \exp \biggl( \sum_{p \les X_j} \frac{\tilde \lambda_f (p)}{\sqrt{p}} W(p, X_j) - \sum_{p \les X_j^{1/2}} \frac{1}{p} W(p^2, X_j) + c_0 + \frac{3}{4} \frac{\log d^2 k^4}{\log X_j}  \biggr) \les \\
k L(1, \chi_{-d}) \exp \biggl( -\frac{3^{I-j-1} D}{16} \log \frac{D}{8} + c_3 + 6 D \cdot 3^{I-j} \max \biggl( 1, \frac{1}{2} \frac{\log d}{\log k} \biggr) \biggr) \les \\ \les k L(1, \chi_{-d}) \exp \biggl( -\frac{3^{I-j-1} D}{32} \log \frac{D}{8} + c_3 \biggr)
\end{multline*} as soon as
\begin{equation} \label{restriction_D}
D \ges  8 \exp \biggl( 600 \max \biggl( 1, \frac{1}{2} \frac{\log d}{\log k} \biggr) \biggr).
\end{equation}

\subsection{The estimation for $f \in \EuScript{F} \cap \EuScript{P}(m)$}
\label{subsec6.5}

In this section, we take into account the contribution from the second Dirichlet polynomial
$$
\sum_{p \les \sqrt{x}} \frac{\lambda_f (p^2)}{p} W(p^2, x).
$$ We will obtain the bound of the form
$$
\sum_m \mathop{{\sum}^{h}}_{f \in \EuScript{F} \cap \EuScript{P}(m)} \exp \biggl( \sum_{p \les x} \frac{\tilde \lambda_f (p)}{\sqrt{p}} W(p, x) + \sum_{p \les \sqrt{x}} \frac{\lambda_f (p^2)}{p} W(p^2, x) \biggr) \ll
k L(1, \chi_{-d}) \bigl( \log x \bigr).
$$

The restriction in the sum over $m$ is $m \les (2 \log 2)^{-1} (\log x)$. The most contribution comes from the subsets $f \in \EuScript{F} \cap \EuScript{P}(m)$ for small $m$. Let us fix a large constant $C_1 > 0$. If $0 \les m < C_1$ we bound the second polynomial trivially:
\begin{equation}\label{smallest_m}
\sum_{p \les \sqrt{x}} \frac{\lambda_f (p^2)}{p} W(p^2, x) \les \sum_{p \les \sqrt{2^{m+1}}} \frac{3}{p} + \sum_{n=m+1}^{+\infty} \frac{1}{n^2} \les C_2.
\end{equation} The same clearly holds true for $f \in \EuScript{F} \cap \EuScript{P}$. Together with the estimate~\eqref{estimate_for_F} obtained in Section~\ref{subsec6.3} it gives the desired bound. \\

Now we consider two more cases: \\

1) $C_1 \les m \les 2 \log \log x$ \\

Using partial summation and Ramanujan bounds $|\tilde \lambda_f (p)| \les 4$, $|\lambda_f (p^2)| \les 3$, we get
\begin{equation} \label{Ramanujan_bound}
\sum_{p \les 2^{m+1}} \frac{\tilde \lambda_f (p)}{\sqrt{p}} W(p, x) + \sum_{p \les 2^{m+1}} \frac{\lambda_f (p^2)}{p} W(p^2, x) \les  c_4 \frac{2^{m/2}}{m}
\end{equation} with an absolute constant $c_4 > 0$. Let $i_m$ be the largest index so that $X_{i_m} \les 2^{m+1}$. We have
\begin{multline*}
\mathop{{\sum}^{h}}_{f \in \EuScript{F} \cap \EuScript{P}(m)} \exp \biggl( \sum_{p \les x} \frac{\tilde \lambda_f (p)}{\sqrt{p}} W(p, x) + \sum_{p \les \sqrt{x}} \frac{\lambda_f (p^2)}{p} W(p^2, x) \biggr) \les \\ 
\exp \biggl(  c_4\frac{2^{m/2}}{m} \biggr) \mathop{{\sum}^{h}}_{f \in \EuScript{F} \cap \EuScript{P}(m)} \exp \biggl( \sum_{2^{m+1} < p \les x} \frac{\tilde \lambda_f (p)}{\sqrt{p}} W(p, x) + \sum_{n > m} \frac{1}{n^2} \biggr) \les \\ 
\exp \biggl(  c_4\frac{2^{m/2}}{m} + c_5 \biggr) \mathop{{\sum}^{h}}_{f \in \EuScript{F} \cap \EuScript{P}(m)} \biggl( \prod_{i=i_m}^I \exp \biggl( \frac{1}{2} \sum_{\substack{X_i < p \les X_{i+1} \\ p > 2^{m+1}}} \frac{\tilde \lambda_f (p)}{\sqrt{p}} W(p, x) \biggr)^2 \biggr).
\end{multline*} Due to the definition of the sets $\EuScript{P}(m)$ there is $z_0 > 2^{m+1}$, so that the last expression does not exceed
$$
\exp \biggl(  c_4\frac{2^{m/2}}{m} + c_5 \biggr) \mathop{{\sum}^{h}}_{f \in \EuScript{F}} \biggl( \prod_{i=i_m}^I \exp \biggl( \frac{1}{2} \sum_{\substack{X_i < p \les X_{i+1} \\ p > 2^{m+1}}} \frac{\tilde \lambda_f (p)}{\sqrt{p}} W(p, x) \biggl)^2 \biggr) \cdot \biggl( \frac{P_m (f, z_0)}{1/m^2} \biggr)^{2M}. 
$$ Repeating the same steps as in the estimation for $f \in \EuScript{F}$, we bound the last expression by
\begin{multline*}
\exp \biggl(  c_4\frac{2^{m/2}}{m} + c_5 \biggr) (1 + \varepsilon_{0,m}) m^{4M} \cdot \\ \mathop{{\sum}^{h}}_{f \in S_k (\Gamma_0 (2))} \biggl( \prod_{i = i_m}^I \sum_{n \les 2 A_i} \frac{1}{n!} \sum_{\substack{X_i < p_1, \ldots, p_n \les X_{i+1} \\ p_l > 2^{m+1}}} \frac{\tilde \lambda_f (p_1) \ldots \tilde \lambda_f (p_n)}{\sqrt{p_1 \ldots p_n}} W(p_1, x) \ldots W(p_n, x) \biggr) \cdot \\ 
\biggl( \sum_{2^m < p_1, \ldots, p_{2M} \les 2^{m+1}} \frac{\lambda_f (p_1^2) \ldots \lambda_f (p_{2M}^2)}{p_1 \ldots p_{2M}} W(p_1^2, z_0) \ldots W(p_{2M}^2, z_0) \biggr),
\end{multline*} where
$$
1 + \varepsilon_{0,m} := \exp \biggl( 2 \sum_{i=i_m}^I \exp \biggl( \frac{B_i}{2} - \log (A_i + 1)! + (A_i + 1) \log \frac{B_i}{2} \biggr) \biggr).
$$ The corresponding subsets of primes in all the sums in the last expression are disjoint. Then, applying the second part of Lemma~\ref{lem5} (with $M_0 = M$), we get the bound
\begin{multline*}
\exp \biggl(  c_4\frac{2^{m/2}}{m} + c_5 \biggr) (1 + \varepsilon_{0,m}) m^{4M} \cdot
3 \frac{k-1}{2\pi^2} \exp \biggl( \sum_{2^{m+1} < p \les x} \frac{1+\chi_{-d} (p)}{p} \biggr) \cdot \\
\sqrt{2\pi} M 4^M \frac{(2M)!}{M!} \biggl( \frac{1}{2^m} \biggr)^M.
\end{multline*} By Stirling formula and Lemma~\ref{lem6}, it does not exceed
$$
k L(1, \chi_{-d}) \exp \biggl( c_4 \frac{2^{m/2}}{m} + 4M \log m + \log \log x + M \log M +
2M + \log M - mM \log 2 + c_6 \biggr).
$$ Choosing $M = \lfloor 2^{m/2} \rfloor$, we bound this by
\begin{equation}\label{small_m}
k L(1, \chi_{-d}) \exp \bigl( - m + \log \log x  \bigr)
\end{equation} since $m \ges  C_1$. To satisfy the restrictions of Lemma~\ref{lem5}, we require
$$
\log \bigl( 2^{m+1} \bigr)^{2M} \les \frac{1}{100} \log \sqrt{k}, 
$$ which is equivalent to
$$  
m \frac{\log 2}{2} \les \bigl(1 - o(1)\bigr) \log \log k  = \bigl(1 - o(1)\bigr) \log \log (x).
$$ This implies
$$
m \les \frac{2}{\log (2)} \log \log x,
$$ which is satisfied by the choice of $m$. Note that the restrictions of Lemma~\ref{lem5} are satisfied since
$$
\biggl( \prod_{i=1}^j X_i^{2A_i} \biggr) X_{j+1}^{2N_{j+1}} (2^{m+1})^{2M} \les \sqrt{k}
$$ due to the choice of the parameters $A_i$'s, $N_{j+1}$ and $M$. \\

2) $m > 2 \log \log x$ \\

Combining Cauchy inequality and Ramanujan bound $|\lambda_f (p^2) | \les 3$, we get
\begin{multline}\label{measure1}
\mathop{{\sum}^{h}}_{f \in \EuScript{F} \cap \EuScript{P}(m)} \exp \biggl( \sum_{p \les x} \frac{\tilde \lambda_f (p)}{\sqrt{p}} W(p, x) + \sum_{p \les \sqrt{x}} \frac{\lambda_f (p^2)}{p} W(p^2, x) \biggr) \les \\ 
\exp \biggl( 3 \log \log (2^{m+1}) + c_7 \biggr) \sqrt{\meas \EuScript{P}(m)} \cdot \biggl( \mathop{{\sum}^{h}}_{f \in \EuScript{F}} \exp \biggl( 2 \sum_{p \les x} \frac{\tilde \lambda_f (p)}{\sqrt{p}} W(p, x) \biggr) \biggr)^{1/2}.
\end{multline} 

To estimate the measure of $\EuScript{P}(m)$, we use Markov inequality and part 2) of Lemma~\ref{lem4}:
\begin{multline} \label{PM_measure_estimate}
\meas \EuScript{P}(m) \les \biggl( \frac{1}{m^2} \biggr)^{-4} \mathop{{\sum}^{h}}_{f \in S_k (\Gamma_0 (2))} \biggl( \sum_{2^m \les p < 2^{m+1}} \frac{\lambda_f (p^2)}{p} W(p^2, z_0) \biggr)^4 \les \\ 
m^8 \mathop{{\sum}^{h}}_{f \in S_k (\Gamma_0 (2))} \sum_{2^m < p_1, p_2, p_3, p_4 \les 2^{m+1}} \frac{\lambda_f (p_1^2) \lambda_f (p_2^2) \lambda_f (p_3^2) \lambda_f (p_4^2)}{p_1 p_2 p_3 p_4} W(p_1^2, z_0) \ldots W(p_4^2, z_0) \les \\ 
m^8 \cdot 2\sqrt{2\pi} \cdot 16 \frac{k-1}{2\pi^2} \cdot 12 \biggl( \frac{1}{2^m} \biggr)^2 \les k \exp \biggl( 8 \log m - 2m \log 2 + c_8 \biggr).
\end{multline} We estimate the second factor in the right hand side of~\eqref{measure1} with the usual approach. Applying to this expression the first part of Lemma~\ref{lem5} with $b_m = 2^m / m!$, we bound the second factor by
$$
\biggl( (1+\varepsilon_0) \cdot 3 \frac{k-1}{2\pi^2} \prod_{i=1}^I \biggl( \sum_{n \les A_i} \frac{4^n}{n!} \biggl( \sum_{X_{i-1} < p \les X_i} \frac{1+\chi_{-d}(p)}{p} \biggr)^n \biggr) \biggr)^{1/2}.
$$ Next, applying the conditional upper bound $L(1, \chi_{-d}) \ll \log \log d$ (see, for example,~\cite{Lamzouri_Li_Sound}), we estimate the right hand side of~\eqref{measure1} by:
\begin{multline} \label{large_m}
\exp \bigl( 3 \log m + c_9 \bigr) \sqrt{k} \exp \bigl( 4 \log m - m \log 2 + c_8 / 2 \bigr) \cdot \\
\sqrt{k} L (1, \chi_{-d})^2 \exp \bigl( 2 \log \log x \bigr) \les k  L (1, \chi_{-d})^2  \exp \bigl( - 0.01 m + \log \log x + c_{10} \bigr) \les \\
k  L (1, \chi_{-d})  \exp \bigl( - 0.01 m + \log \log x + \log \log \log d + c_{11} \bigr).
\end{multline} Thus, combining \eqref{smallest_m}, \eqref{small_m} and \eqref{large_m}, we obtain
\begin{multline*}
\biggl( \mathop{{\sum}^{h}}_{f \in \EuScript{F} \cap \EuScript{P}} + \sum_{0 \les m \les \log x /  (2 \log 2) } \mathop{{\sum}^{h}}_{f \in \EuScript{F} \cap \EuScript{P}(m)} \biggr) \exp \biggl( \sum_{p \les x} \frac{\tilde \lambda_f (p)}{\sqrt{p}} W(p, x) + \\
\sum_{p \les \sqrt{x}} \frac{\lambda_f (p^2)}{p} W(p^2, x) \biggr) \les
2C_1 k L(1, \chi_{-d}) \exp \bigl( C_2 + \log \log x \bigr) + \\
\sum_{m \ges C_1} k L(1, \chi_{-d}) \exp \bigl( -m + \log \log x \bigr) + \\ 
\sum_{m \ges  2 \log \log x} k L(1, \chi_{-d}) \exp \bigl( -0.01 m + \log \log x + \log \log \log d + c_{11} \bigr) \les \\
k L(1, \chi_{-d}) \exp \bigl(\log \log x + c_{12} \bigr)
\end{multline*} since $d$ is bounded by a power of $x$. So we get
\begin{multline*}
\mathop{{\sum}^{h}}_{f \in \EuScript{F}} \exp \bigl( \log \bigl| L_1 (f) L_2 (f) \bigr| \bigr) \les \\
k L(1, \chi_{-d}) \exp \biggl( \log \log x + c_{12} - \sum_{p \les \sqrt{x}} \frac{W(p^2, x)}{p} + c_0 + \frac{3}{4} \frac{\log d^2 k^4}{\log x} \biggr).
\end{multline*} The sum over $p \les x$ in the right hand side of the last inequality can trivially be estimated from below as
$$
\sum_{p \les x} \frac{W(p^2, x)}{p} \ges \log \log x - c_{13}.
$$ Finally, we get
\begin{equation} \label{final_bound_for_F}
\mathop{{\sum}^{h}}_{f \in \EuScript{F}} \exp \bigl( \log \bigl| L_1 (f) L_2 (f) \bigr| \bigr) \les k L(1, \chi_{-d}) \exp \biggl( c_{12} + c_{13} + c_0 + \frac{3}{4} \frac{\log d^2 k^4}{\log x} \biggr).
\end{equation}

\subsection{The estimation for $f \in S(j) \cap \EuScript{P}(m)$}
\label{subsec6.6}

The upper bound for the sum over $f \in S(j) \cap \EuScript{P}(m)$ is obtained in a similar way. Here we combine the approaches from sections~\ref{subsec6.4} and~\ref{subsec6.5}. 

First, assume $0 \les m < C_1$. Combining together the bounds~\eqref{smallest_m} and~\eqref{goal_section_sj}, we get
\begin{multline} \label{Sj_smallest_m}
	\biggl( \mathop{{\sum}^{h}}_{f \in S (j) \cap \EuScript{P}} + \sum_{0 \les m < C_1} \mathop{{\sum}^{h}}_{f \in S(j) \cap \EuScript{P}(m)} \biggr) \exp \biggl( \sum_{p \les X_j} \frac{\tilde \lambda_f (p)}{\sqrt{p}} W(p, X_j) + \\
	\sum_{p \les X_j^{1/2}} \frac{\lambda_f(p^2)-1}{p} W(p^2, X_j) +
	c_0 + \frac{3}{4} \frac{\log d^2 k^4}{\log X_j} \biggr) \ll kL(1, \chi_{-d}) \exp \biggl( -\frac{3^{I-j-1}D}{32} \log \frac{D}{8} \biggr).
\end{multline} Now consider the remaining two cases: \\

1) $C_1 \les m \les 2 \log \log X_j$ \\

Again, let $i_m$ be the largest index so that $X_{i_m} \les 2^{m+1}$. Applying~\eqref{Ramanujan_bound}, we have
\begin{multline*}
\mathop{{\sum}^{h}}_{f \in S(j) \cap \EuScript{P}(m)} \exp \biggl( \sum_{p \les X_j} \frac{\tilde \lambda_f (p)}{\sqrt{p}} W(p, X_j) + \sum_{p \les X_j^{1/2}} \frac{\lambda_f (p^2)}{p} W(p^2, X_j) \biggr) \les \\ 
\exp \biggl(  c_4\frac{2^{m/2}}{m} \biggr) \mathop{{\sum}^{h}}_{f \in S(j) \cap \EuScript{P}(m)} \exp \biggl( \sum_{2^{m+1} < p \les X_j} \frac{\tilde \lambda_f (p)}{\sqrt{p}} W(p, X_j) + \sum_{n > m} \frac{1}{n^2} \biggr) \les \\ 
\exp \biggl(  c_4\frac{2^{m/2}}{m} + c_5 \biggr) \mathop{{\sum}^{h}}_{f \in S(j) \cap \EuScript{P}(m)} \biggl( \prod_{i=i_m}^j \exp \biggl( \frac{1}{2} \sum_{\substack{X_i < p \les X_{i+1} \\ p > 2^{m+1}}} \frac{\tilde \lambda_f (p)}{\sqrt{p}} W(p, X_j) \biggr)^2 \biggr) \cdot \\
\biggl( \frac{G_{j+1} (f, y_0)}{B_{j+1}} \biggr)^{2N_{j+1}}.
\end{multline*} By definition of $\EuScript{P}(m)$, there is $z_0 > 2^{m+1}$, so that it does not exceed
\begin{multline*}
\exp \biggl(  c_4\frac{2^{m/2}}{m} + c_5 \biggr) \mathop{{\sum}^{h}}_{f \in S(j)} \biggl( \prod_{i=i_m}^I \exp \biggl( \frac{1}{2} \sum_{\substack{X_i < p \les X_{i+1} \\ p > 2^{m+1}}} \frac{\tilde \lambda_f (p)}{\sqrt{p}} W(p, X_j) \biggr)^2 \biggr) \cdot \\
\biggl( \frac{G_{j+1} (f, y_0)}{B_{j+1}} \biggr)^{2N_{j+1}} \cdot \biggl( \frac{P_m (f, z_0)}{1/m^2} \biggr)^{2M}. 
\end{multline*} By the standard argument this expression can be bounded by
\begin{multline*}
\exp \biggl(  c_4\frac{2^{m/2}}{m} + c_5 \biggr) (1 + \varepsilon_{0,m}^{(j)}) m^{4M} B_{j+1}^{-2N_{j+1}} \cdot \\ 
\mathop{{\sum}^{h}}_{f \in S_k (\Gamma_0 (2))} \biggl( \prod_{i = i_m}^j \sum_{n \les 2 A_i} \frac{1}{n!} \sum_{\substack{X_i < p_1, \ldots, p_n \les X_{i+1} \\ p_l > 2^{m+1}}} \frac{\tilde \lambda_f (p_1) \ldots \tilde \lambda_f (p_n)}{\sqrt{p_1 \ldots p_n}} W(p_1, x) \ldots W(p_n, x) \biggr) \cdot \\
\biggl( \sum_{X_j < p_1, \ldots, p_{2N_{j+1}} \les X_{j+1}} \frac{\tilde \lambda_f (p_1) \ldots \tilde \lambda_f (p_{2N_{j+1}})}{\sqrt{p_1 \ldots p_{2N_{j+1}}}} W(p_1, y_0) \ldots W(p_{2N_{j+1}}, y_0) \biggr) \cdot \\
\biggl( \sum_{2^m < p_1, \ldots, p_{2M} \les 2^{m+1}} \frac{\lambda_f (p_1^2) \ldots \lambda_f (p_{2M}^2)}{p_1 \ldots p_{2M}} W(p_1^2, z_0) \ldots W(p_{2M}^2, z_0) \biggr),
\end{multline*} where
$$
1 + \varepsilon_{0,m}^{(j)} := \exp \biggl( 2 \sum_{i=i_m}^j \exp \biggl( \frac{B_i}{2} - \log (A_i + 1)! + (A_i + 1) \log \frac{B_i}{2} \biggr) \biggr).
$$ The corresponding subsets of primes in all the sums in the last expression are disjoint. Then, applying the second part of Lemma~\ref{lem5} (with $M_0 = M$, $N_0 = N_{j+1}$), we get the bound
\begin{multline*}
\exp \biggl(  c_4\frac{2^{m/2}}{m} + c_5 \biggr) (1 + \varepsilon_{0,m}^{(j)}) m^{4M} B_{j+1}^{-2N_{j+1}} \cdot
3 \frac{k-1}{2\pi^2} \exp \biggl( \sum_{2^{m+1} < p \les x} \frac{1+\chi_{-d} (p)}{p} \biggr) \cdot \\
\frac{(2N_{j+1})!}{N_{j+1}!} \biggl( \sum_{X_j < p \les 2X_{j+1}} \frac{1+\chi_{-d}(p)}{p} \biggr)^{N_{j+1}} \cdot \sqrt{2\pi} M 4^M \frac{(2M)!}{M!} \biggl( \frac{1}{2^m} \biggr)^M.
\end{multline*} By Stirling formula and Lemma~\ref{lem6}, it does not exceed
\begin{multline*}
k L(1, \chi_{-d}) \exp \biggl( -2N_{j+1} \log B_{j+1} + N_{j+1} \log N_{j+1} + c_1' N_{j+1} + \log N_{j+1} + \log \log X_{j+1} + \\ 
c_4 \frac{2^{m/2}}{m} + 4M \log m + \log \log x + M \log M +
2M + \log M - mM \log 2 + c_6' \biggr).
\end{multline*} Choosing $M = \lfloor 2^{m/2} \rfloor$, and taking all the same parameters $A_i$'s, $B_{j+1}$, $N_{j+1}$, and $D$ as in Section~\ref{subsec6.4}, we bound the last expression by
$$
k L(1, \chi_{-d}) \exp \biggl( - m - \frac{3^{I-j-D}}{16} \log \frac{D}{8} + \log \log X_j + c_2' \biggr).
$$ Then, finally, for $m \les 2 \log \log X_j$,
\begin{multline} \label{Sj_small_m}
\mathop{{\sum}^{h}}_{f \in S(j) \cap \EuScript{P}(m)} \exp \biggl( \sum_{p \les X_j} \frac{\tilde \lambda_f (p)}{\sqrt{p}} W(p, X_j) + \sum_{p \les X_j^{1/2}} \frac{\lambda_f(p^2)-1}{p} W(p^2, X_j) + \\
c_0 + \frac{3}{4} \frac{\log d^2 k^4}{\log X_j} \biggr) \ll kL(1, \chi_{-d}) \exp \biggl( -m - \frac{3^{I-j-1} D}{32} \log \frac{D}{8} \biggr).
\end{multline} \\

2) $m > 2 \log \log X_j$ \\

Applying Cauchy inequality and Ramanujan bound, we obtain
\begin{multline} \label{measure_2}
\mathop{{\sum}^{h}}_{f \in S(j) \cap \EuScript{P}(m)} \exp \biggl( \sum_{p \les X_j} \frac{\tilde \lambda_f (p)}{\sqrt{p}} W(p, X_j) + \sum_{p \les \sqrt{X_j}} \frac{\lambda_f (p^2)}{p} W(p^2, X_j) \biggr) \les \\ 
\exp \biggl( 3 \log \log (2^{m+1}) + c_7 \biggr) \sqrt{\meas \EuScript{P}(m)} \cdot \biggl( \mathop{{\sum}^{h}}_{f \in S(j)} \exp \biggl( 2 \sum_{p \les X_j} \frac{\tilde \lambda_f (p)}{\sqrt{p}} W(p, X_j) \biggr) \biggr)^{1/2}.
\end{multline} 

The measure of $\EuScript{P}(m)$ is estimated in~\eqref{PM_measure_estimate}
$$
\meas \EuScript{P}(m) \les k \exp \biggl( 8 \log m - 2m \log 2 + c_8 \biggr).
$$ To treat the second factor in the right hand side of~\eqref{measure_2}, we proceed as in Section~\ref{subsec6.4}. It is bounded by
$$
\biggl( \mathop{{\sum}^{h}}_{f \in S(j)} \exp \biggl( \sum_{p \les X_j} \frac{\tilde \lambda_f (p)}{\sqrt{p}} W(p, X_j) \biggr)^2 \biggl( \frac{G_{j+1} (f, y_0)}{B_{j+1}} \biggr)^{2N_{j+1}} \biggr)^{1/2}.
$$ Applying to this expression the second part of Lemma~\ref{lem5} with $N_0 = N_{j+1}$, $M_0 = 0$, we estimate it by
\begin{multline*}
\biggl( (1+\varepsilon_0^{(j)}) B_{j+1}^{-2N_{j+1}} \cdot 3 \frac{k-1}{2\pi^2} \prod_{i=1}^j \biggl( \sum_{n \les A_i} \frac{4^n}{n!} \biggl( \sum_{X_{i-1} < p \les X_i} \frac{1+\chi_{-d}(p)}{p} \biggr)^n \biggr) \cdot \\
\frac{(2N_{j+1})!}{N_{j+1}!} \biggl( \sum_{X_j < p \les X_{j+1}} \frac{1+\chi_{-d}(p)}{p} \biggr)^{N_{j+1}} \biggr)^{1/2}.
\end{multline*} 

Then similarly to~\eqref{large_m}, and with the same choice of parameters $A_i$, $B_{j+1}$, $N_{j+1}$, and $D$ as earlier, we bound the right hand side of~\eqref{measure_2} by
\begin{multline*} 
\exp \bigl( 3 \log m + c_9 \bigr) \sqrt{k} \exp \biggl( 4 \log m - m \log 2 + \frac{c_8}{2} \biggr) \cdot \\
\sqrt{k} L (1, \chi_{-d})^2 \exp \biggl( 2 \log \log X_j - \frac{3^{I-j-1}D}{32} \log \frac{D}{8} + \frac{c_2}{2} \biggr) \les \\ 
k  L (1, \chi_{-d})^2  \exp \biggl( - 0.01 m + \log \log X_j - \frac{3^{I-j-1}D}{32} \log \frac{D}{8} + c_{10}' \biggr) \les \\
k  L (1, \chi_{-d})  \exp \biggl( - 0.01 m + \log \log X_j + \log \log \log d - \frac{3^{I-j-1}D}{32} \log \frac{D}{8} + c_{11}' \biggr).
\end{multline*} Then, for $m > \log \log X_j$,
\begin{multline} \label{Sj_large_m}
\mathop{{\sum}^{h}}_{f \in S(j) \cap \EuScript{P}} \exp \biggl( \sum_{p \les X_j} \frac{\tilde \lambda_f (p)}{\sqrt{p}} W(p, X_j) + \sum_{p \les X_j^{1/2}} \frac{\lambda_f (p^2)}{p} W(p^2, X_j) \biggr) \les \\
k L(1, \chi_{-d}) \exp \biggl( -0.01 m - \frac{3^{I-j-1}D}{32} \log \frac{D}{8} + \log \log X_j + \log \log \log d + c_{11}' \biggr).
\end{multline}

Thus, combining~\eqref{Sj_smallest_m}, \eqref{Sj_small_m} and~\eqref{Sj_large_m}, we obtain
\begin{multline*}
\biggl( \mathop{{\sum}^{h}}_{f \in S(j) \cap \EuScript{P}} + \sum_{0 \les m \les \log X_j /  (2 \log 2) } \mathop{{\sum}^{h}}_{f \in S(j) \cap \EuScript{P}(m)} \biggr) \exp \biggl( \sum_{p \les X_j} \frac{\tilde \lambda_f (p)}{\sqrt{p}} W(p, X_j) + \\
\sum_{p \les X_j^{1/2}} \frac{\lambda_f (p^2)-1}{p} W(p^2, X_j) +
c_0 + \frac{3}{4} \frac{\log d^2 k^4}{\log X_j} \biggr) \les \\
2C_1 k L(1, \chi_{-d}) \exp \biggl( -\frac{3^{I-j-1}D}{32} \log \frac{D}{8} \biggr) + \\
\sum_{m \ges C_1} k L(1, \chi_{-d}) \exp \biggl( -m - \frac{3^{I-j-1}D}{32} \log \frac{D}{8}  \biggr) + 
\sum_{m \ges  2 \log \log X_j} k L(1, \chi_{-d}) \cdot \\ 
\exp \biggl( -0.01 m - \frac{3^{I-j-1}D}{32} \log \frac{D}{8} + \log \log \log d + c_{12}' + \frac{3}{4} \frac{\log d^2 k^4}{\log X_j}  \biggr) \les \\
k L(1, \chi_{-d}) \exp \biggl( - c_{13}' \frac{3^{I-j-1}D}{64} \log \frac{D}{8} \biggr)
\end{multline*} for the appropriate constant $c_{13}' > 0$. Finally, we get
\begin{equation} \label{final_bound_for_Sj}
\mathop{{\sum}^{h}}_{f \in S(j)} \exp \bigl( \log \bigl| L_1 (f) L_2 (f) \bigr| \bigr) \les k L(1, \chi_{-d}) \exp \biggl( - c_{13}' \frac{3^{I-j-1} D}{64} \log \frac{D}{8} + c_{14} \biggr).
\end{equation}

\subsection{The estimation for $f \in S(0)$}
\label{subsec6.7}

To estimate the contribution coming from $f \in S(0)$ we need the following crude upper bound for the second moment:
\begin{equation} \label{Sound_bound}
\mathop{{\sum}^{h}}_{f \in S(0)} \exp \bigl( 2 \log \bigl|L_1 (f) L_2 (f) \bigr| \bigr) \les k (\log k)^7 L(1, \chi_{-d})^2,
\end{equation} which is obtained by Soundararajan's method. The proof of this bound is given in Section~\ref{subsec6.8}.

By Cauchy inequality,
$$
\mathop{{\sum}^{h}}_{f \in S(0)} \exp \bigl( \log \bigl|L_1 (f) L_2 (f) \bigr| \bigr) \les \sqrt{ \meas S(0)} \biggl( \mathop{{\sum}^{h}}_{f \in S(0)} \exp \bigl( 2 \log \bigl|L_1 (f) L_2 (f) \bigr| \bigr) \biggr)^{1/2}.
$$ The second factor can be bounded by~\eqref{Sound_bound}. Then we get
\begin{equation} \label{whole_sum_S0}
\mathop{{\sum}^{h}}_{f \in S(0)} \exp \bigl( \log \bigl|L_1 (f) L_2 (f) \bigr| \bigr) \les \sqrt{ \meas S(0) } \biggl( k (\log k)^7 L (1, \chi_{-d})^2 \biggr)^{1/2}.
\end{equation} The first factor is estimated by Markov inequality together with the first part of Lemma~\ref{lem4}:
\begin{multline*}
\meas S(0) \les B_1^{-2 N_1} \mathop{{\sum}^{h}}_{f \in S_k (\Gamma_0 (2))} \biggl( \sum_{p \les X_1} \frac{\tilde \lambda_f (p)}{\sqrt{p}} W(p, X_1) \biggr)^{2 N_1} \les \\ 
\exp \bigl(  -  2N_1 \log B_1 \bigr) \cdot 2 \frac{k-1}{2\pi^2} \frac{(2N_1)!}{N_1!} \biggl( \sum_{p \les X_1} \frac{1+\chi_{-d}(p)}{p} \biggr)^{N_1}.
\end{multline*} By Stirling formula, the right hand side of~\eqref{whole_sum_S0} does not exceed
\begin{equation} \label{whole_sum_S0_2}
k (\log k)^{7/2} L(1, \chi_{-d}) \exp \biggl( - N_1 \log B_1 + \frac{1}{2} N_1 \log N_1 + c_{15} N_1 + \frac{N_1}{2} \log \log \log x \biggr).
\end{equation} Choose
$$
N_1 = \biggl\lfloor \frac{B_1}{2} \biggr\rfloor = \biggl\lfloor \frac{D}{240\sqrt{3}} e^{I-1} \biggr\rfloor.
$$ This choice meets the restrictions of Lemmas~\ref{lem4} and~\ref{lem5}. Then~\eqref{whole_sum_S0_2} does not exceed
$$
k L(1, \chi_{-d}) \exp \biggl( \frac{7}{2} \log \log k - \frac{D}{10000} \log \log x \biggr) \les
k L(1, \chi_{-d}) \exp\bigl( -\log \log k \bigr)
$$ since $\log \log x = (1 + o(1)) \log \log k$. Thus, we get
\begin{equation} \label{final_bound_for_S0}
\mathop{{\sum}^{h}}_{f \in S(0)} \exp \bigl( \log \bigl|L_1 (f) L_2 (f) \bigr| \bigr) \les k L(1, \chi_{-d}) \exp \bigl( -\log \log k \bigr).
\end{equation}

Combining~\eqref{final_bound_for_S0} together with final bounds for $\EuScript{F}$~\eqref{final_bound_for_F} and $S(j)$~\eqref{final_bound_for_Sj} we finally get
\begin{multline*}
\mathop{{\sum}^{h}}_{f \in S_k (\Gamma_0 (2))} \exp \bigl( \log \bigl| L_1 (f) L_2 (f) \bigr| \bigl) \les
k L(1, \chi_{-d}) \biggl\{  \exp \biggl( c_{12} + c_{13} + c_0 + \frac{3}{4} \frac{\log d^2 k^4}{\log x} \biggr) + \\
\sum_{j=1}^{I-1} \exp \biggl( - c_{13}' \frac{3^{I-j-1} D}{64} \log \frac{D}{8} + c_{14} \biggr) + \exp \bigl( -\log \log k \bigr) \biggr\} \les \\
k L(1, \chi_{-d}) \exp \biggl( c_{16} + 6D \max \biggl( 1, \frac{1}{2} \frac{\log d}{\log k}  \biggr) \biggr).
\end{multline*} Given the restriction~\eqref{restriction_D} we take
$$
D = 8 \exp\biggl( 600 \max \biggl( 1, \frac{1}{2} \frac{\log d}{\log k} \biggr) \biggr).
$$ Thus,
\begin{multline*}
\mathop{{\sum}^{h}}_{f \in S_k (\Gamma_0 (2))} \exp \bigl( \log \bigl| L_1 (f) L_2 (f) \bigr| \bigl) \les \\
k L(1, \chi_{-d}) \exp\biggl( c_{16} + 48 \max \biggl( 1, \frac{1}{2} \frac{\log d}{\log k}  \biggr) \exp \biggl( 600 \max \biggl( 1, \frac{1}{2} \frac{\log d}{\log k}  \biggr) \biggr) \biggr)
\end{multline*} as desired. To complete the proof of Theorem~\ref{thm2} we only need to show~\eqref{Sound_bound}.

\subsection{Proof of~\eqref{Sound_bound}}
\label{subsec6.8}

We write the second moment as follows:
$$
\mathop{{\sum}^{h}}_{f \in S_k (\Gamma_0 (2))} \bigl| L_1 (f) L_2 (f) \bigr|^2 = \int_{-\infty}^{+\infty} e^{2V} \meas \biggl\{ f \in S_k (\Gamma_0 (2)) : \log \bigl| L_1 (f) L_2 (f) \bigr| > V  \biggr\} dV.
$$ Next, split the last expression into the sum of two integrals
$$
I_1 + I_2 := \int_{-\infty}^{3 \log \log k} e^{2V} \meas \bigl\{ \ \ldots \  \bigr\} dV + \int_{3 \log \log k}^{+\infty} e^{2V} \meas \bigl\{ \ \ldots \ \bigr\} dV. 
$$ The first integral can be estimated directly:
$$
I_1 \les \frac{k-1}{2\pi^2} \int_{-\infty}^{3 \log \log k} e^{2V} dV \les k (\log k)^6 \les \frac{k}{2} (\log k)^7 L(1, \chi_{-d})^2.
$$ The last inequality follows from the assumptions of Theorem~\ref{thm2}, namely, $k \gg d^{\varepsilon}$, and conditional lower bound for $L(s, \chi_{-d})$ (see, for example,~\cite{Littlewood1928}):
$$
L(1, \chi_{-d}) \gg \frac{1}{\log \log d}.
$$

The measure in the second integral is evaluated in the usual way. Fix the number $E$ so that
$$
E > 40 \max \biggl( 1, \frac{\log d}{\log k} \biggr),
$$ and let $x = k^{2E/V}$. By Lemma~\ref{lem1}, we have
\begin{multline*}
\log \bigl|L_1 (f) L_2 (f) \bigr| \les \sum_{p \les x} \frac{\tilde \lambda_f (p)}{\sqrt{p}} W(p, x) + \\
\sum_{p \les \sqrt{x}} \frac{\lambda_f (p^2) - 1}{p} W(p^2, x) +
c_0 + \frac{3V}{E} \max \biggl( 1, \frac{\log d}{\log k} \biggr).
\end{multline*} Then
\begin{multline*}
\meas \biggl\{ f \in S_k (\Gamma_0 (2)) : \log \bigl|L_1 (f) L_2 (f) \bigr| > V \biggr\} \les \\
\meas \biggl\{ f \in S_k (\Gamma_0 (2)) : \sum_{p \les x} \frac{\tilde \lambda_f (p)}{\sqrt{p}} W(p, x) > \frac{V}{4} \biggr\},
\end{multline*} since $V > 3 \log \log x = 3 (1 - o(1)) \log \log k$, 
$$
\sum_{p \les x} \frac{ \lambda_f (p^2)}{p} < 3 \bigl(1 + o(1) \bigr) \log \log x,
$$ and
$$
\frac{3V}{E} \max \biggl( 1, \frac{\log d}{\log k} \biggr) < \frac{V}{13}.
$$ Combining Markov inequality and~Lemma~\ref{lem4}, we get
\begin{multline*}
\mathop{{\sum}^{h}}_{f \in S_k (\Gamma_0 (2))} \biggl(\frac{V}{4}\biggr)^{-2N} \biggl( \sum_{p\les x} \frac{\tilde \lambda_f (p)}{\sqrt{p}}  \biggr)^{2N} \les \\
2 \frac{k-1}{2\pi^2} \frac{(2N)!}{N!} \exp \biggl( -2N \log \frac{V}{4} + \frac{3}{2} N \log \log \log x  \biggr).
\end{multline*} With the choice $N = \lfloor V / (8E) \rfloor$ the restrictions of~Lemma~\ref{lem4} are satisfied, and the right hand side in the last inequality does not exceed
$$
k L(1, \chi_{-d}) \exp \biggl( -\frac{V}{20E} \log \log V \biggr) \les k L(1, \chi_{-d}) \exp \bigl( -10 V \bigr).
$$ Hence,
$$
I_2 \les k L(1, \chi_{-d}) \int_{3 \log \log k}^{+\infty} e^{2V - 10 V} dV \les k L(1, \chi_{-d}),
$$ which completes the proof.

\section*{Acknowledgments}

I thank Maksym Radziwi{\l\l} for introducing me to this problem and much useful advice. I also thank Alex de Faveri, Peter Zenz, Mayank Pandey, and Liyang Yang for helpful discussions, and the referees for the careful reading of the paper.

\nocite{*}
\bibliographystyle{abbrv}
\bibliography{Linnik_8_bib}

\end{document}